\documentclass[11pt]{article}
\usepackage[letterpaper]{geometry}
\usepackage{amsmath,amsthm,amsfonts,amssymb}
\usepackage{enumerate,color,xcolor}
\usepackage{graphicx}
\usepackage{subfigure}
\usepackage{url}
\usepackage{natbib}

\numberwithin{equation}{section}
\theoremstyle{plain}

\newcommand{\E}{\mathbb{E}}

\newtheorem{lemma}{Lemma}
\newtheorem{theorem}{Theorem}
\newtheorem{proposition}{Proposition}

\newtheorem{remark}{Remark}

\newtheorem{corollary}{Corollary}
\newtheorem{assumption}{Assumption} 
 
\begin{document}
 % Start your text
\begin{center}
  \Large \bf  Optimal Unbiased Estimation for Expected Cumulative Discounted Cost
  \end{center}

\author{}
\begin{center}
Zhenyu Cui\,\footnote{School of Business, Stevens Institute of Technology, Hoboken, NJ-07030, United States of America;
    zcui6@stevens.edu},
    Michael C. Fu\,\footnote{R.H. Smith School of Business, University of Maryland, College Park, MD-20742, United States of America;
    mfu@umd.edu},
    Yijie Peng\,\footnote{Department of Industrial Engineering and Management, Peking University, Beijing, 100871, People's Republic of China;
    pengyijie@pku.edu.cn},
          Lingjiong Zhu\,\footnote{Department of Mathematics, Florida State University, 1017 Academic Way, Tallahassee, FL-32306, United States of America; zhu@math.fsu.edu.}
\end{center}

\begin{center}
 \today
\end{center}

\begin{abstract}
We consider estimating an expected infinite-horizon cumulative discounted cost/reward contingent on an underlying stochastic process by Monte 
Carlo simulation. An unbiased estimator based on truncating the cumulative cost at a random horizon is proposed. Explicit forms for 
the optimal distributions of the random horizon are given, and explicit expressions for the optimal random truncation level are obtained, leading to a full analysis of the bias-variance tradeoff when comparing this new class of randomized estimators with traditional fixed truncation estimators. Moreover, we characterize when the optimal randomized estimator is 
preferred over a fixed truncation estimator by considering the tradeoff between bias and variance. This comparison provides guidance on when to choose randomized estimators over fixed truncation estimators in practice. Numerical experiments substantiate 
the theoretical results.

\end{abstract}

\vspace{0.5cm}

%\noindent{\textit{Keywords:} Simulation; unbiased estimation; simulation optimization;  computing budget allocation; cumulative costs}

%\noindent{\textit{Subject classifications:} simulation optimization; efficiency;  computational complexity}

\newpage

\section{Introduction\label{s1}}

\emph{Motivation and Problem Formulation.}  
We consider estimating an expected cumulative cost 
\begin{equation*}
\alpha:=\mathbb{E}\left[\int_0^{\infty} g(X_t,t)dt\right],
\end{equation*} 
where $X:=\{X_{t},t\geq 0\}$ is the underlying stochastic process defined
on a metric space $\mathcal{X}$ (e.g. $\mathcal{X}=\mathbb{R}$ or $\mathbb{R}^{d}$),
and $g:\mathcal{X}\times\mathbb{R}_{+}\rightarrow\mathbb{R}$ is a real-valued
function of both the underlying process and time.
This covers the special case $g(X_t, t):=e^{-c t}f(X_t)$, which is frequently used in asset pricing,
where $c>0$ is the discount factor, 
and $\alpha$ is referred to as cumulative discounted cost.
For example, $\alpha$ can be the expected present value of a discounted cumulative cash flow contingent on the future value of an underlying asset, where $c>0$ is the discount factor, $f(X_{t})$ the cash flow rate
contingent on the asset value $X_{t}$ at time $t$.

In finance, this is related to simulating a cumulative (discounted) cash flow of a  stochastic perpetuity (\cite{fox1989simulating,blanchet2011exact}) or a mortgage-backed security (MBS) (\cite{glasserman2003resource}). In steady-state simulation, $\alpha$ corresponds to the expected long-run behavior of the sample time-average, e.g., average waiting time in a queueing system (\cite{whitt2002stochastic}). In project management, this corresponds to the accumulated present value of a project, and is a useful metric to decide between capital projects. This generic form also appears in the optimal lifetime consumption problem studied by \cite{merton1969lifetime}. 

In our setting,  $\alpha$ is assumed to be unavailable in closed form, but Monte Carlo simulation  can be used to estimate $\alpha$. It is computationally infeasible to simulate the cumulative cost over an infinite horizon. Thus, a truncation technique is needed to estimate $\alpha$, and batching is typically used to construct a confidence interval (\cite{alexopoulos2016spsts}). However, truncating  at a fixed horizon  generally leads to bias, which is difficult to quantify in statistical inference. We propose a randomized estimator that truncates at a random horizon to retrieve the unbiasedness. By doing so, an asymptotically valid confidence interval can be obtained by  sampling i.i.d. sample paths of cumulative cost truncated at the random horizon, which can be justified by a classical central limit theorem. 
Since variability is introduced by the random horizon, the unbiasedness of the estimator may come at the cost of a larger variance, which motivates us to ask the following question:
\begin{quote}
\textit{ What is the optimal randomized unbiased estimator, and in what sense is it optimal?}
\end{quote}

Throughout, ``optimal" means that within the class of unbiased estimators constructed based on a randomized truncation level, we are searching for the estimator satisfying some optimal criteria. More specifically, we consider three such criteria common in comparing Monte Carlo estimators: minimizing the variance  with a linear penalty on the computational cost, minimizing variance with a fixed cost, and minimizing the work-variance product (see \cite{glynn1992asymptotic}). 

The proposed estimator truncates the cumulative cost at a random horizon following a distribution independent of the underlying stochastic process. Our goal is to find an optimal distribution for the random horizon.
We consider both a constrained optimization problem where the variance of the estimator is minimized subject to a fixed expected simulation cost/work, and an unconditional optimization problem where we seek to minimize the work-variance product (see e.g. \cite{glynn1992asymptotic}) of the estimator. These are infinite-dimensional functional optimization problems over all possible distributions, which are difficult to solve numerically in general. However, we derive explicit forms for the optimal distributions of the random horizon by using the maximum principle for an optimal control problem (see e.g. \cite{bryson1975applied}). The optimal distributions are in a shifted distribution class.
For a discounted continuous cumulative cost contingent on an exponential L\'{e}vy process, the optimal randomization distributions are shifted exponential distributions. \cite{glynn1983randomized} considers a similar cumulative cost estimation problem; however, since the objective there is to minimize the {\it asymptotic} variance of a randomized (unbiased) estimator, the focus is on asymptotic results, whereas we focus on the finite simulation budget setting. 
In particular, our results show that the optimal randomization distribution in the setting of \cite{glynn1983randomized}  is not necessarily optimal in any of the fixed computational budget settings we consider.

Although the proposed randomized estimator eliminates bias, it inevitably increases the variance. We define a utility function as a linear combination of bias and variance. With a positive weight on the variance, we show that the optimal randomized estimator is less favorable than the fixed truncation estimator when the computational budget is sufficiently small. 
A threshold function of the computational budget for the weight of variance is provided, and the advantage of the optimal randomized estimator can be justified if the weight of the variance is less than this threshold.

\emph{Related Literature.}  
Related work on using randomization in other settings to recover unbiasedness includes 
\cite{mcleish2011general,rhee2015unbiased} for stochastic differential equation (SDE) models. \cite{rhee2015unbiased} use an infinite sum truncated at a random horizon independent of the underlying SDE to obtain an unbiased estimator of the path functionals associated with the SDE. 
They also derive an optimal distribution for the random horizon, but their objective is to minimize the {\it asymptotic} variance of the estimator, which assumes the computational budget goes to infinity, whereas we consider the setting of a fixed (finite) computational budget. 
As a result, solving for the optimal randomization distribution in their setting leads to a discrete optimization problem, whereas our formulation leads to a continuous-time functional optimal control problem that can be solved analytically by applying the maximum principle.
 Other work using randomization to eliminate bias includes unbiased estimation of Markov chain equilibrium expectations (\cite{glynn2014exact}), unbiased stochastic optimization (\cite{blanchet2015unbiased2}), unbiased Bayesian inference (\cite{lyne2015russian}), and unbiased maximum likelihood inference (\cite{jacob2015nonnegative}).  
  None of the previous work provides analysis comparing randomized unbiased Monte Carlo (RUMC) method with traditional Monte Carlo (MC), e.g., a fixed truncation estimator, by taking both bias and variance into account. There is some recent work studying the choice of the optimal randomization distribution in the setting of unbiased estimation involving the solution of a SDE. \cite{cui2019optimal} proposed an approach based on the dual formulation, and \cite{kahale2019optimal} proposed a method based on convex hulls. As far as we know, this is the first work dealing with RUMC from an optimal control perspective. 
 
Closely related to the RUMC method is the multi-level Monte Carlo  (MLMC) method introduced in  \cite{giles2008multilevel,giles2015multilevel}, which combines biased estimators of different step sizes to improve the convergence rate of traditional MC method.
There is also recent interest in combining RUMC and MLMC for developing an unbiased MLMC method,  which can be found in  \cite{vihola2015unbiased,zheng2017clt}.

\emph{Contribution.} 
The contributions of our paper are three-fold:
 \begin{enumerate}
 \item We propose an optimal randomized unbiased estimator for estimating the expected (infinite) cumulative cost/reward in a fixed computational budget setting. 

 \item We provide explicit forms for the optimal randomization distributions balancing the trade-off between variance and computational cost. 
 
 \item We offer theoretical justification for the advantage/disadvantage of RUMC over traditional MC by explicitly considering the tradeoff between bias and variance. 
\end{enumerate}

MLMC focuses on improving convergence rate rather than eliminating bias, and the existing RUMC work only minimizes the asymptotic variance of the randomized estimator.
	To the best of the authors' knowledge, this is the first work resulting in explicit forms for the optimal distribution of an unbiased randomized estimator and comparing randomization with fixed truncation taking both bias and variance into consideration.
 In an example for the general class of exponential L\'{e}vy processes, 
 the explicit forms can be calculated analytically, and the condition that the optimal RUMC outperforms traditional MC is given analytically.
The explicit structure of the distribution function of the optimal random truncation level is particularly useful for ``post-estimation" analysis, and allows us to carry out a full diagnosis of the bias-variance tradeoff.

\emph{Organization.}  
The remainder of the paper is organized as follows. Section \ref{s2} proposes the randomized unbiased estimator with three optimal randomization distributions. We compare the optimal randomized estimator and fixed truncation estimator in Section \ref{s3}.  Numerical experiments are presented in Section \ref{s5}. Section \ref{s6}  concludes the paper.

\section{Unbiased Randomized Estimator\label{s2}}

We consider estimating the cumulative cost/reward:
$\alpha:= \E\left[\int_0^{\infty} g(X_s,s)ds\right]$, 
which is assumed to be well-defined and finite. Here $X:=\{ X_s: s\geq 0 \}$ is the underlying stochastic process. 

Let $N$ be a random horizon independent of the underlying stochastic process $X$, and $Q$ be the distribution of $N$, which satisfies
$$Q\in \mathcal{M}(\mathbb{R}^{+}):=\{ Q: ~ Q(N\geq 0)=1,~ Q(N>s)>0,~s\in[0,\infty) \}.$$
The proposed randomized unbiased estimator is
\begin{equation}
I:=\int_{0}^{\infty}g(X_s, s)\frac{1_{\{N>s\}}}{Q(N>s)}ds=\int_{0}^{N}\frac{g(X_s, s)}{Q(N>s)}ds.
\end{equation}
Note that here in the second expression there is the indicator random variable $1_{\{N>s\}}$, and for the third expression the random variable $N$ appears in the upper integration range. Thus we can name this estimator $I$ also as a random truncation estimator.

Due to the independence of $N$ and $X$, the unbiasedness of the proposed estimator can be established straightforwardly by applying Fubini's theorem:
\begin{equation}
\mathbb{E}\left[\int_{0}^{\infty}g(X_s,s)\frac{1_{\{N>s\}}}{Q(N>s)}ds\right]
=\int_{0}^{\infty}\mathbb{E}[g(X_s,s)]ds=:\alpha.\notag
\end{equation}

In the main body of the paper, we ignore the technicality induced by possible discretization for simulating the continuous-time cost process. %The discretized cost process is explicitly treated in the  Appendix.
We consider the following three  optimization problems arising from RUMC:
\begin{enumerate}
\item minimize the variance of the estimator subject to a linear penalty on the computational cost/work;
\item minimize the variance of the estimator subject to a fixed pre-specified level of computational cost/work;
\item minimize the work-variance product of the estimator.
\end{enumerate}

Previous work on RUMC considered three similar optimization problems with the variance replaced by asymptotic variance (\cite{rhee2015unbiased}). 
Large (small) computational cost/work corresponds to small (large) bias.
The three optimizations are natural formulations of the tradeoff between bias and variance.
It turns out that solving the first optimization problem helps solving the succeeding optimization problem(s). Throughout the paper, we assume the expectations and variances of all estimators  are finite and well-defined to avoid the problems of interest becoming meaningless.

\subsection{Minimizing Variance with Penalty}\label{mvp}

We want to optimize over all possible distributions $Q\in\mathcal{M}(\mathbb{R}^{+})$ in order to minimize the variance of the estimator with a penalty for the computational cost:
\begin{equation}
\inf_{Q\in\mathcal{M}(\mathbb{R}^{+})}\left\{\mbox{Var}\left(\int_{0}^{\infty}g(X_s,s)\frac{1_{\{N>s\}}}{Q(N>s)}ds\right)
+\lambda\mathbb{E}^{Q}[N]\right\},~ \lambda>0.\label{prob1}
\end{equation}
This is inherently an infinite-dimensional optimization problem. 
 The following lemma rewrites the optimization problem (\ref{prob1}) in a form that is more amenable for analysis.

\begin{lemma} The optimization problem \eqref{prob1} is equivalent to
\begin{equation}
\inf_{Q\in\mathcal{M}(\mathbb{R}^{+})}
\left\{2\int_{0}^{\infty}\frac{\Gamma(s)}{Q(N>s)}ds+\lambda\int_{0}^{\infty}Q(N>s)ds\right\},~ \lambda>0,\label{pr-cont}
\end{equation}
where
$$\Gamma(s):=\int_{s}^{\infty}\mathbb{E}[g(X_{t},t)g(X_{s},s)]dt.$$
\end{lemma}

\begin{proof}
Notice that
\begin{equation}
\mbox{Var}\left(\int_{0}^{\infty}g(X_s,s)\frac{1_{\{N>s\}}}{Q(N>s)}ds\right)
=\mathbb{E}\left[\left(\int_{0}^{\infty}g(X_s,s)\frac{1_{\{N>s\}}}{Q(N>s)}ds\right)^{2}\right]
-\left(\int_{0}^{\infty}\mathbb{E}[g(X_s,s)]ds\right)^{2}.\notag
\end{equation}
By Fubini's theorem and the independence between $X$ and $N$, we have
\begin{align*}
\mathbb{E}\left[\left(\int_{0}^{\infty}g(X_s,s)\frac{1_{\{N>s\}}}{Q(N>s)}ds\right)^{2}\right]
&=2\int_{0}^{\infty}\int_{0}^{t}
\mathbb{E}\left[g(X_s,s)g(X_t,t)\frac{1_{\{N>s\}}1_{\{N>t\}}}{Q(N>s)Q(N>t)}\right]dsdt\notag
\\
&=2\int_{0}^{\infty}\int_{0}^{t}
\mathbb{E}\left[g(X_s,s)g(X_t,t)\frac{1_{\{N>t\}}}{Q(N>s)Q(N>t)}\right]dsdt\notag
\\
&=2\int_{0}^{\infty}\int_{0}^{t}\frac{\mathbb{E}[g(X_s,s)g(X_t,t)]}{Q(N>s)}dsdt
\nonumber
\\
&=2\int_{0}^{\infty}\int_{s}^{\infty}\frac{\mathbb{E}[g(X_s,s)g(X_t,t)]}{Q(N>s)}dtds
=2\int_{0}^{\infty}\frac{\Gamma(s)}{Q(N>s)}ds.
\label{calc1}
\end{align*}
Since $\int_{0}^{\infty}\mathbb{E}[g(X_s,s)]ds$ is independent of $Q$, we can drop it in the optimization, which leads to the conclusion.
\end{proof}
%where $\Gamma(s)=\int_{s}^{\infty}\mathbb{E}[g(X_s,s)g(X_t,t)]dt\geq 0$.
Following a similar procedure in the proof of Lemma 1, we have  $\mathbb{E}\left[\left(\int_{0}^{\infty}g(X_s,s)ds\right)^{2}\right]=2\int_{0}^{\infty}\Gamma(s)ds$.
The random truncation increases the variance by noticing $\int_{0}^{\infty}\Gamma(s)ds\leq \int_{0}^{\infty}\frac{\Gamma(s)}{Q(N>s)}ds$, but the increased variance is compensated by a decreased computational cost. The objective of the optimization problem (\ref{pr-cont}) is a functional of $\Gamma(s)$, $\lambda$, and $Q$. Thus, we expect the optimal randomization distribution $Q^{*}$ for the optimization problem (\ref{pr-cont})  should be determined by $\Gamma(s)$ and $\lambda$. Intuitively, $\Gamma(s)$ captures how fast the cost process $g(X_t,t)$ decays after time $s$, while $\lambda$ is the unit cost for computing the cumulative cost.

\begin{assumption}\label{as1}
 $\Gamma(s)$  is a non-negative and strictly decreasing smooth function.
\end{assumption}

The non-negativity and monotonicity in Assumption \ref{as1} can be justified if the cost process $g(X_s,s)$ is non-negative (or non-positive) and $\mathbb{E}[g(X_{t},t)g(X_{s},s)]$ is non-increasing in $s$.  In the case where the cost process has both positive and negative parts, we can decompose it into the difference of two non-negative processes and estimate the cumulative cost of both processes separately. Under Assumption 1, we have an explicit form for the optimal distribution given in the following theorem.

\begin{theorem}\label{main1}
Under Assumption \ref{as1}, for the optimization problem (\ref{pr-cont}), 
\begin{equation}
Q^{\ast}(N>s)=
\begin{cases}
1 &\mbox{for $s\leq s^{\ast}$},
\\
\sqrt{\frac{2\Gamma(s)}{\lambda}} &\mbox{for $s>s^{\ast}$},
\end{cases}
\end{equation}
where $s^*=\inf\{s\in[0,\infty):~ \Gamma(s)\leq \lambda/2\}$, and the minimum is given by
\begin{align*}
\inf_{Q\in\mathcal{M}(\mathbb{R}^{+})}
\left\{2\int_{0}^{\infty}\frac{\Gamma(s)}{Q(N>s)}ds+\lambda\int_{0}^{\infty}Q(N>s)ds\right\}
=2\int_0^{s^{\ast}} \Gamma(s) ds +2\sqrt{2\lambda} \int_{s^{\ast}}^{\infty} \sqrt{\Gamma(s)} ds +\lambda s^{\ast}.
\end{align*}
\end{theorem}

\begin{proof} For the optimization problem (\ref{pr-cont}),  we have
\begin{equation*}
Q^{\ast}(N>s)=\arg\min_x L(x;s,\lambda),
\end{equation*}
where
\begin{equation}
L(x;s,\lambda):=2\frac{\Gamma(s)}{x}+\lambda x\notag.
\end{equation}
 Notice that the function $ L(x;s,\lambda)$ decreases for $x<\sqrt{\frac{2\Gamma(s)}{\lambda}}$
and increases for $x>\sqrt{\frac{2\Gamma(s)}{\lambda}}$.
In addition, for any $Q\in\mathcal{M}(\mathbb{R}^{+})$,
$Q(N>s)$ is required to decrease from $1$ to $0$ as $s$ goes from $0$ to $\infty$. Then, we can calculate
\begin{align*}
\inf_{Q\in\mathcal{M}(\mathbb{R}^{+})}
\left\{2\int_{0}^{\infty}\frac{\Gamma(s)}{Q(N>s)}ds+\lambda\int_{0}^{\infty}Q(N>s)ds\right\}
&=
2\int_{0}^{\infty}\frac{\Gamma(s)}{Q^{\ast}(N>s)}ds+\lambda\int_{0}^{\infty}Q^{\ast}(N>s)ds\notag\\
&=2\int_0^{s^{\ast}} \Gamma(s) ds +2\sqrt{2\lambda} \int_{s^{\ast}}^{\infty} \sqrt{\Gamma(s)} ds +\lambda s^{\ast}. 
\end{align*}
 Combining the arguments above leads to the conclusion.
\end{proof}

\begin{remark} \label{remark1} It is straightforward to establish that if  $\Gamma(0)\leq \lambda/2$, then $s^*=0$ and if  $\Gamma(0)>\lambda/2$, then $s^*\in(0,\infty)$ and $\Gamma(s^*)=\lambda/2$.
We can see the following insight from the explicit form of the optimal distribution $Q^*$: large $\lambda$ and small $\Gamma(s)$ correspond to small $Q^*(N>s)$, which means the distribution of the random truncation concentrates on the domain where $N$ is small. This insight intuitively makes sense. Large unit cost $\lambda$ for computing the cumulative cost favors a small truncation size. Small $\Gamma(s)$ roughly indicates that the cost process  decays fast after time $s$, which thus encourages us to put more computational effort before time $s$. If $\Gamma(s)$ is non-monotone, the optimal $Q^{*}$ can be solved by an optimal control problem, which can be found in the appendix.\end{remark}

In the following, we consider the exponential  L\'{e}vy process, which includes geometric Brownian motion (GBM) as a special case; see e.g. \cite{fu2016option}. Let $Y_t$ be a L\'{e}vy process with the characteristic triplet $(\mu,\sigma^2, \nu)$, and its characteristic exponent is given by $\phi(\cdot)$, which is uniquely characterized by the L\'{e}vy-Khintchine formula: $\mathbb{E}[e^{\beta Y_t}]=e^{t \phi(\beta)}$. Then, $X_{t}=e^{Y_{t}}$ is an exponential L\'{e}vy process. 
Let $g(X_{t},t)=e^{-ct}f(X_{t})$, where
$f(x)=x^{\beta}$ for a fixed $\beta$, so  $g(X_t,t)=e^{-ct}X_t^{\beta}$.
%with moment generating function $\mathbb{E}[X_{t}^{\beta}]=\mathbb{E}[e^{\beta Y_{t}}]=e^{t\phi(\beta)}$,

Define
\begin{equation}
\phi_1(\beta):=\phi(\beta)-c, \quad \phi_2(\beta):=\phi(\beta)-2c, \label{def-phi}
\end{equation}
and assume that $\phi_1(\beta)<0$ and $\phi_2(2\beta)<0$ in order for the relevant  integrals to be well-defined.
%and they shall appear in later explicit calculations related to the exponential L\'{e}vy processes.
Then, we have
\begin{equation*}
\mathbb{E}[g(X_t,t)g(X_s,s)]
=\mathbb{E}[e^{-c(t+s)+\beta(Y_t+Y_s)}]
=\mathbb{E}[e^{-c(t-s)+\beta (Y_t-Y_s)}] \mathbb{E}[e^{-2cs+2\beta Y_s}]
=e^{(t-s)\phi_1(\beta)+s\phi_2(2\beta)},\notag
\end{equation*}
and
\begin{equation}
\Gamma(s)=\int_{s}^{\infty}\mathbb{E}[g(X_t,t)g(X_s,s)]dt\\
=\int_{s}^{\infty}e^{(t-s)\phi_1(\beta)+s\phi_2(2\beta)}dt=\frac{1}{|\phi_1(\beta)|}e^{-s|\phi_2(2\beta)|}.\notag
\end{equation}

%It is clear that $\Gamma(s)$ is decreasing in $s$, and from Remark \ref{special1} we have that $\Gamma(s)=\Gamma(s)$. In this situation, it is a convention that we keep the notation $\Gamma(s)$ in computing related quantities.

 %Since we consider the discount factor, we observe that we can absorb it into the L\'{e}vy process. Introduce a new L\'{e}vy process $\widetilde{Y}_t$ with characteristic triplet $(\mu-c/\beta,\sigma^2, \nu)$, and its characteristic exponent is denoted by $\widetilde{\phi}(\cdot)$, then we have that $g(X_t,t)=e^{-ct+\beta Y_t}=e^{\beta \widetilde{Y}_t}$.

\begin{corollary}\label{expLevy1}
If $\{X_t\}$ is  an exponential L\'evy process with characteristic exponent $\phi$ and $f(x)=x^{\beta}$, then  
under the optimal randomization distribution $Q^{\ast}$, $N$ is a shifted exponential random variable with the probability density function given by
\begin{equation}
q^{\ast}(s):=
|\phi_2(2\beta)|\sqrt{\frac{1}{2\lambda|\phi_1(\beta)|}}e^{-\frac{1}{2}|\phi_2(2\beta)|s}1_{\{s>s^{\ast}\}},\notag
\end{equation}
where the optimal shift $s^{\ast}$ is given by:
\begin{equation}
s^{\ast}:=
\begin{cases}
0 &\mbox{if $\lambda\geq\frac{2}{|\phi_1(\beta)|}$},\notag
\\
-\frac{1}{|\phi_2(2\beta)|}\log\left(\frac{1}{2}\lambda|\phi_1(\beta)|\right) &\mbox{if $\lambda<\frac{2}{|\phi_1(\beta)|}$},\notag
\end{cases}
\end{equation}
where $\phi_1$ and $\phi_2$ are given by (\ref{def-phi}). 
\end{corollary}

\begin{proof}
Note that for $\lambda\geq 2\Gamma(0)=\frac{2}{|\phi_1(\beta)|}$,
$s^{\ast}=0$. Otherwise, $\Gamma(s^{\ast})=\frac{1}{2}\lambda$,  so that
\begin{equation}
s^{\ast}=-\frac{1}{|\phi_2(2\beta)|}\log\left(\frac{1}{2}\lambda|\phi_1(\beta)|\right).\notag
\end{equation}
We conclude that the optimal $Q^{\ast}$ is given by
\begin{equation}
Q^{\ast}(N>s):=
\begin{cases}
1 &\mbox{for $0\leq s\leq s^{\ast}$},\notag
\\
\sqrt{\frac{2}{\lambda|\phi_1(\beta)|}}e^{-\frac{1}{2}|\phi_2(2\beta)|s} &\mbox{for $s>s^{\ast}$},\notag
\end{cases}
\end{equation}
which completes the proof by differentiation. 
\end{proof}

%%%%%%%%%%%%%%%%%%%%%%%%%%%%%%%%%%%%%%%%%%%%%%%%%%%%%%%%%
\subsection{Constrained Optimization}

In this section, we consider the second optimization problem, in which we minimize the variance of the randomized estimator given that the  computational budget is {\it fixed} at a level $m$:
\begin{equation}\label{op1}
\inf_{Q\in\mathcal{M}(\mathbb{R}^{+}): \mathbb{E}^{Q}[N]=m}\left\{\mbox{Var}\left(\int_{0}^{\infty}g(X_s,s)\frac{1_{\{N>s\}}}{Q(N>s)}ds\right)\right\}.
\end{equation}

 An explicit characterization for the optimal distribution is obtained in the following theorem by using the maximum principle of an optimal control problem.

\begin{theorem}\label{main2} Under Assumption \ref{as1}, for the optimization problem (\ref{op1}),\vspace{2mm}
	
	(i) If  $m> \int_{0}^{\infty
	}\sqrt{\Gamma(u)}du/\sqrt{\Gamma(0)}$, then 
	\begin{equation*}
	Q^{\ast}(N>s):=
	\begin{cases}
	1 &\mbox{for $ s\leq s^{\ast}$},
	\\
	\sqrt{\frac{\Gamma(s)}{\Gamma(s^*)}} &\mbox{for $s>s^{\ast}$},
	\end{cases}
	\end{equation*}
	%where $s^{\ast}=\Gamma^{-1}(\lambda/2)$ and $\lambda$ satisfies the following algebraic equation
	where $s^{\ast}$ is the unique positive solution to the following equation:
	\begin{align*}
	s^\ast +\frac{\int_{s^\ast}^\infty {\sqrt{ \Gamma(u)} du}}{\sqrt{\Gamma(s^\ast)}}&=m,
	\end{align*}
	and the minimum variance is given by
	\begin{align*}
	\inf_{Q\in\mathcal{M}(\mathbb{R}^{+}): \mathbb{E}^{Q}[N]=m}\left\{\mbox{Var}\left(\int_{0}^{\infty}g(X_s,s)\frac{1_{\{N>s\}}}{Q(N>s)}ds\right)\right\}
	=2\int_0^{s^{\ast}}\Gamma(s) ds+2\Gamma(s^{\ast})(m-s^{\ast})-\alpha^2.
	\end{align*}
	
	(ii) If $m\leq \int_{0}^{\infty
	}\sqrt{\Gamma(u)}du/\sqrt{\Gamma(0)}$, then
	\begin{align}\label{opt1}
	Q^*(N>s)&:=\frac{m\sqrt{\Gamma(s)}}{\int_{0}^{\infty}\sqrt{\Gamma(u)}du},
	\end{align}
	and the minimum variance is given by
	\begin{align*}
	\inf_{Q\in\mathcal{M}(\mathbb{R}^{+}): \mathbb{E}^{Q}[N]=m}\left\{\mbox{Var}\left(\int_{0}^{\infty}g(X_s,s)\frac{1_{\{N>s\}}}{Q(N>s)}ds\right)\right\}
	=\frac{2}{m}\left(\int_0^{\infty}\sqrt{\Gamma(u)} du\right)^2-\alpha^2.
	\end{align*}
\end{theorem}

\begin{proof}
Consider the following infinite horizon optimal control problem:
\begin{equation}\label{oc}
\begin{aligned}
\inf_{\{u(s)\in[0,1]: ~s\in[0,\infty)\}}&\left\{\int_{0}^{\infty}\frac{2\Gamma(s)}{u(s)}ds\right\},
\\s.t.\quad & \dot{z}(s)=u(s),\quad s\in[0,\infty),
\\
&z(0)=0,\quad \lim_{t\to\infty}z(t)=m.
\end{aligned}
\end{equation}
We introduce the Hamiltonian:
$$H(z(s),u(s),p(s),s):=\frac{2\Gamma(s)}{u(s)}+p(s) u(s),$$
where $p(s)$ is an adjoint variable.
For $s\in[0,\infty)$, the optimal control $u^*(t)$ satisfies the following maximum principle, see e.g. \cite{halkin1974necessary}:
\begin{equation*}
\begin{cases}
&\mbox{Optimal condition}:~  u^*(s)=\inf_{u\in[0,1]} H(z(s),u,p(s),s), \\
%& H(z(s),u^*(s),\lambda(s),s)=0,\\
& \mbox{Adjoint equation}:~ \dot{p}(s)=-H_z(z(s),u(s),p(s),s)=0.
\end{cases}
\end{equation*}
  From the adjoint equation, we know there exists $\gamma\in\mathbb{R}$ such that $p(s)\equiv \gamma$ for $s\in[0,\infty)$. For $\gamma\leq 0$, control $u^*(s)\equiv 1$ on $[0,\infty)$ satisfies the optimal condition, but it cannot satisfy the state constraint in (\ref{oc}). Thus, we have $\gamma\in\mathbb{R^+}$.
  As in the proof of Theorem \ref{main1},  the optimal condition implies
\begin{equation*}
u^*(s)=
\begin{cases}
1 &\mbox{for $s\leq s^{\ast}$},
\\
\sqrt{\frac{2\Gamma(s)}{\gamma}} &\mbox{for $s>s^{\ast}$},
\end{cases}
\end{equation*}
 where $s^*=\inf\{s\in[0,\infty):~ \Gamma(s)\leq \gamma/2\}$.  %Thus,  $\nu>0$ must hold for the Lagrangian multiplier satisfying Pontryagin's maximum principle if $\tau\neq m$.
By the state constraint in (\ref{oc}),
 \begin{equation}
\lim_{t\to\infty}z(t)=\int_{0}^{\infty}u^*(s)ds=s^{\ast}+\int_{s^{\ast}}^{\infty}\sqrt{\frac{2\Gamma(s)}{\gamma}}ds=m,\notag
 \end{equation}
we have $s^*<m$ and
\begin{equation}
\gamma=\frac{2}{(m-s^{\ast})^{2}}\left(\int_{s^{\ast}}^{\infty}\sqrt{\Gamma(u)}du\right)^{2}.\notag
\end{equation}
From Remark \ref{remark1}, we know that if  $\Gamma(0)\leq \gamma/2$, then $s^*=0$, which implies
\begin{equation}
\Gamma(0)\leq \frac{1}{m^{2}}\left(\int_{0}^{\infty
}\sqrt{\Gamma(u)}du\right)^{2};\notag
\end{equation}
if $\Gamma(0)>\gamma/2$, then $s^*\in(0,\infty)$ and $\Gamma(s^*)=\gamma/2$, which implies
\begin{equation}
\Gamma(s^{\ast})=\frac{1}{(m-s^{\ast})^{2}}\left(\int_{s^{\ast}}^{\infty}\sqrt{\Gamma(u)}du\right)^{2},\notag
\end{equation}
or equivalently,
\begin{align*}
s^\ast +\frac{\int_{s^\ast}^\infty {\sqrt{ \Gamma(s)} ds}}{\sqrt{\Gamma(s^\ast)}}=m.
\end{align*}
Define
$$G(s):=s +\frac{\int_{s}^\infty {\sqrt{ \Gamma(s)} ds}}{\sqrt{\Gamma(s)}}-m.$$
We have
$G(m)>0$, and
\begin{align*}
G'(s)=\frac{-\Gamma^{\prime}(s) \int_{s}^{\infty} \sqrt{\Gamma(u)}du}{2\Gamma(s)\sqrt{\Gamma(s)}}>0.
\end{align*}
Thus, equation $G(s)=0$ has a unique solution on $(0,\infty)$ if and only if $G(0)<0$, or equivalently,
\begin{equation}
\Gamma(0)>\frac{1}{m^{2}}\left(\int_{0}^{\infty
}\sqrt{\Gamma(u)}du\right)^{2}.\notag
\end{equation}
 Summarizing the above arguments, the maximum principle and state constraint in (\ref{op1}) offer a unique  $u^*(s)$ on $[0,\infty)$,  which is the optimal control.

Under Assumption \ref{as1}, the optimal control $u^*(s)$ is non-increasing on $[0,\infty)$ and $\lim_{t\to\infty}u^{*}(t)=0$. By noticing that the  optimization (\ref{op1}) is equivalent to
\begin{equation*}
\inf_{Q\in\mathcal{M}(\mathbb{R}^{+}): \mathbb{E}^{Q}[N]=m}\left\{\int_{0}^{\infty}\frac{2\Gamma(s)}{Q(N>s)}ds\right\},
\end{equation*}
and
\begin{align*}
\mathbb{E}^{Q}[N]=\int_{0}^{\infty} Q(N>s) ds=m,
\end{align*}
we know that $Q^*(N>s)=u^*(s)$ on $[0,\infty)$ is the optimal distribution for the optimization problem (\ref{op1}).
The rest of the proof is a straightforward calculation.
\end{proof}
\begin{remark}\label{remark2} Optimization problem (\ref{prob1}) can also be viewed as an optimal control problem but without a state constraint in (\ref{oc}), which is imposed by the computational budget constraint.
When the computational budget is smaller than a threshold, i.e., $m\leq \int_{0}^{\infty
}\sqrt{\Gamma(u)}du/\sqrt{\Gamma(0)} $, we have
 $s^*=0$, so that the distribution $Q^*$ given by (\ref{opt1}) is supported on $\mathbb{R}^{+}$.
Increasing the computational budget $m$ on the range $(0, \int_{0}^{\infty
}\sqrt{\Gamma(u)}du/\sqrt{\Gamma(0)} )$ would make the tail of the distribution $Q^*$ heavier, which indicates that the distribution of the optimal randomization shifts more weight toward the domain when $N$ is large as the computational budget $m$ increases.  When the computational budget is larger than a threshold, i.e., $m> \int_{0}^{\infty
}\sqrt{\Gamma(u)}du/\sqrt{\Gamma(0)}$, we have $s^*>0$, which indicates that the truncation size $N$ would be almost surely larger than a certain threshold under the optimal randomization if the computational budget is larger than a certain threshold.
 \end{remark}

As an illustration, we then show the optimal distribution for the optimization (\ref{op1}) when $X_{t}$ is an exponential L\'{e}vy process.

\begin{corollary}\label{expL}
If $\{X_t\}$ is  an exponential L\'evy process with characteristic exponent $\phi$ and $f(x)=x^{\beta}$, then when $m|\phi_2(2\beta)|\leq 2$, the optimal distribution $Q^{\ast}$ is given by
\begin{equation}
Q^{\ast}(N>s)
=\frac{m}{2}|\phi_2(2\beta)|e^{-\frac{s}{2}|\phi_2(2\beta)|} \notag
\end{equation}
for any $0<s<\infty$.
On the other hand, when $m|\phi_2(2\beta)|>2$,
the optimal $Q^{\ast}$ is given by
\begin{equation*}
Q^{\ast}(N>s):=
\begin{cases}
1 &\mbox{for $ s\leq m-\frac{2}{|\phi_2(2\beta)|}$},
\\
e^{\frac{m}{2}|\phi_2(2\beta)|-1-\frac{s}{2}|\phi_2(2\beta)|} &\mbox{for $s>m-\frac{2}{|\phi_2(2\beta)|}$},
\end{cases}
\end{equation*}
where  $\phi_2$ is given by (\ref{def-phi}). 
\end{corollary}

\begin{proof}
Let us recall that $\Gamma(s)=\frac{1}{|\phi_1(\beta)|}e^{-s|\phi_2(2\beta)|}$, and we have
\begin{equation}
\int_{0}^{\infty}\sqrt{\Gamma(u)}du=
\frac{2}{\sqrt{|\phi_1(\beta)|}\cdot|\phi_2(2\beta)|}.\notag
\end{equation}
Therefore, when
\begin{equation}
\frac{1}{|\phi_1(\beta)|}=\Gamma(0)\leq\frac{1}{m^{2}}\left(\int_{0}^{\infty}\sqrt{\Gamma(u)}du\right)^{2}
=\frac{4}{m^{2}|\phi_1(\beta)|\cdot|\phi_2(2\beta)|^{2}},\notag
\end{equation}
the optimal $Q^{\ast}$ is given by
\begin{equation}
Q^{\ast}(N>s)
=\frac{m\sqrt{\Gamma(s)}}{\int_{0}^{\infty}\sqrt{\Gamma(u)}du}
=\frac{m}{2}|\phi_2(2\beta)|e^{-\frac{s}{2}|\phi_2(2\beta)|},~s\in(0,\infty).\notag
\end{equation}
 When $m|\phi_2(2\beta)|>2$,
the optimal $Q^{\ast}$ is given by
\begin{equation}
Q^{\ast}(N>s):=
\begin{cases}
1 &\mbox{for $ s\leq s^{\ast}$},\notag
\\
\sqrt{\frac{2}{\gamma}}\frac{1}{\sqrt{|\phi_1(\beta)|}}e^{-\frac{s}{2}|\phi_2(2\beta)|} &\mbox{for $s>s^{\ast}$},\notag
\end{cases}
\end{equation}
where $s^{\ast}=\Gamma^{-1}(\gamma/2)=\frac{-1}{|\phi_2(2\beta)|}\log(\frac{1}{2}\gamma|\phi_1(\beta)|)$ and
\begin{align*}
\gamma&=\frac{2}{(m-s^{\ast})^{2}}\left(\int_{s^{\ast}}^{\infty}\sqrt{\Gamma(u)}du\right)^{2}
\\
&=\frac{2}{\left(m+\frac{1}{|\phi_2(2\beta)|}\log(\frac{1}{2}\gamma|\phi_1(\beta)|)\right)^{2}}
\frac{1}{|\phi_1(\beta)|}\frac{4}{|\phi_2(2\beta)|^{2}}e^{-s^{\ast}|\phi_2(2\beta)|}
\nonumber
\\
&=\frac{2}{\left(m+\frac{1}{|\phi_2(2\beta)|}\log(\frac{1}{2}\gamma|\phi_1(\beta)|)\right)^{2}}
\frac{1}{|\phi_1(\beta)|}\frac{4}{|\phi_2(2\beta)|^{2}}\frac{1}{2}\gamma|\phi_1(\beta)|,
\nonumber
\end{align*}
which implies that
\begin{equation}
\gamma=\frac{2}{|\phi_1(\beta)|}e^{2-m|\phi_2(2\beta)|},\notag
\end{equation}
and thus
\begin{equation}
s^{\ast}=\Gamma^{-1}(\gamma/2)=\frac{-1}{|\phi_2(2\beta)|}\log\left(\frac{1}{2}\gamma|\phi_1(\beta)|\right)
=m-\frac{2}{|\phi_2(2\beta)|}.\notag
\end{equation}
This completes the proof. 
\end{proof}

%%%%%%%%%%%%%%%%%%%%%%%%%%%%%%%%%%%%%%%%%%%%%%%%%%%%%%%%%%%
\subsection{Minimization of the  Work Variance Product}

In this section, we consider the third optimization  problem, which is to minimize the product
of the variance and the expected value of $N$, i.e.,
\begin{equation}
\inf_{Q\in\mathcal{M}(\mathbb{R}^{+})}\left\{\mbox{Var}\left(\int_{0}^{\infty}g(X_s,s)\frac{1_{\{N>s\}}}{Q(N>s)}ds\right)
\cdot\mathbb{E}^{Q}[N]\right\}.\label{prob3}
\end{equation}

A key observation is that this optimization problem is equivalent to first minimizing
over the variance conditional on $\mathbb{E}^{Q}[N]=m$ and then minimizing over all possible values
of fixed levels $m\geq 0$.  The main idea is that we can first  conditional on the value of $\mathbb{E}^{Q}[N]$, and then exhaust all possible values of $\mathbb{E}^{Q}[N]$ to search for the optimum.  We have the following equivalence in the two optimization problems:\small
\begin{align}
&\inf_{Q\in\mathcal{M}(\mathbb{R}^{+})}\left\{\mbox{Var}\left(\int_{0}^{\infty}g(X_s,s)\frac{1_{\{N>s\}}}{Q(N>s)}ds\right)
\cdot\mathbb{E}^{Q}[N]\right\}\notag
\\
=&\inf_{m\geq 0}\left\{m\cdot
\inf_{Q\in\mathcal{M}(\mathbb{R}^{+}): \mathbb{E}^{Q}[N]=m}\left\{\mbox{Var}\left(\int_{0}^{\infty}g(X_s,s)\frac{1_{\{N>s\}}}{Q(N>s)}ds\right)\right\}\right\}.
\nonumber
\end{align}\normalsize
or equivalently by plugging in the corresponding expressions, i.e.,
\begin{equation}
\inf_{m\geq 0}\left\{m\cdot
\left[2\int_{0}^{\infty}\frac{\Gamma(s)}{Q^{\ast}(N>s)}ds
-\left(\int_{0}^{\infty}\mathbb{E}[g(X_s,s)]ds\right)^{2}\right]\right\}.\notag
\end{equation}

We first establish the following lemma before proving the main result of this section.

\begin{lemma}\label{amazing}
Under Assumption \ref{as1}, 
$$\int_0^{\infty} \Gamma(s)ds >\frac{\alpha^2}{2}.$$
\end{lemma}

\begin{proof}
By definition, we have
\begin{align*}
2\int_{0}^{\infty}\Gamma(s)ds
&=2\int_{0}^{\infty}\int_{s}^{\infty}\mathbb{E}[g(X_s,s)g(X_t,t)]dtds
\\=&2\int_{0}^{\infty}\int_{0}^{t}\mathbb{E}[g(X_s, s)g(X_t,t)]dsdt
=\mathbb{E}\left[\left(\int_{0}^{\infty}g(X_s, s)ds\right)^{2}\right]\\
=& \left(\mathbb{E}\left[\int_{0}^{\infty}g(X_s, s)ds\right]\right)^{2}  + Var\left(\int_{0}^{\infty}g(X_s, s)ds\right)>\alpha^2,
\end{align*}
noticing that $\alpha=\mathbb{E}\left[\int_{0}^{\infty}g(X_s, s)ds\right]$.
This completes the proof. 
\end{proof}

\begin{theorem}\label{theo2} Under Assumption \ref{as1},
 for the optimization problem (\ref{prob3}), 
\begin{equation*}
Q^{\ast}(N>s):=
\begin{cases}
1 &\mbox{for $s\leq s^{\ast\ast}$},
\\
\sqrt{\frac{\Gamma(s)}{\Gamma(s^{**})}} &\mbox{for $s>s^{\ast\ast}$},
\end{cases}
\end{equation*}
where  $s^{\ast\ast}$ is the unique positive solution to the following equation:
\begin{align*}
\frac{\alpha^2}{2}+s^{\ast\ast}\Gamma(s^{\ast\ast})-\int_0^{s^{\ast\ast}} \Gamma(s) ds=0.
\end{align*}
	The minimum  value of the work-variance product is given by
\begin{align*}
\inf_{Q\in\mathcal{M}(\mathbb{R}^{+})}\left\{\mbox{Var}\left(\int_{0}^{\infty}g(X_s,s)\frac{1_{\{N>s\}}}{Q(N>s)}ds\right)
\mathbb{E}^{Q}[N]\right\}
=2(m^{\ast\ast})^2  \Gamma(s^{\ast\ast}),
\end{align*}
where
\begin{align*}
m^{\ast\ast}&=s^{\ast\ast} +\frac{\int_{s^{\ast\ast}}^{\infty} \sqrt{\Gamma(u)}du }{\sqrt{\Gamma(s^{\ast\ast})}}.
\end{align*}
\end{theorem}

\begin{proof}
We have 
\begin{align*}
&\inf_{m\geq 0}\left\{m\cdot
\left[2\int_{0}^{\infty}\frac{\Gamma(s)}{Q^{\ast}(N>s)}ds
-\left(\int_{0}^{\infty}\mathbb{E}[g(X_s,s)]ds\right)^{2}\right]\right\}\notag
\\
&=\min\bigg\{\inf_{0\leq m\leq\frac{\int_{0}^{\infty}\sqrt{\Gamma(u)}du}{\sqrt{\Gamma(0)}}}
\left\{2m\int_{0}^{\infty}\frac{\Gamma(s)}{Q^{\ast}(N>s)}ds
-m\left(\int_{0}^{\infty}\mathbb{E}[g(X_s,s)]ds\right)^{2}\right\},
\\
&
\inf_{m>\frac{\int_{0}^{\infty}\sqrt{\Gamma(u)}du}{\sqrt{\Gamma(0)}}}
\left\{2m\int_{0}^{\infty}\frac{\Gamma(s)}{Q^{\ast}(N>s)}ds-m\left(\int_{0}^{\infty}\mathbb{E}[g(X_s,s)]ds\right)^{2}\right\}\bigg\}
\\
&=\inf_{s^*\in[0,\infty)}\left(\frac{\int_{s^*}^{\infty}\sqrt{\Gamma(y)}dy}{\sqrt{\Gamma(s^*)}}+s^*\right)
\cdot\left[ 2\int_0^{s^*} \Gamma
(y)dy +2\sqrt{\Gamma(s^*)} \int_{s^*}^{\infty} \sqrt{\Gamma(y)}dy -\alpha^2  \right],
\end{align*}
where the second equality is justified by Theorem \ref{main2}, and we recall that
\begin{align*}
 m =s^\ast +\frac{\int_{s^\ast}^\infty {\sqrt{ \Gamma(s)} ds}}{\sqrt{\Gamma(s^\ast)}}.
\end{align*}
Let
\begin{align*}
K(s^*):=&\left(\frac{\int_{s^*}^{\infty}\sqrt{\Gamma(u)}du}{\sqrt{\Gamma(s^*)}}+s^*\right)
\cdot\left[ 2\int_0^{s^*} \Gamma(s)ds +2\sqrt{\Gamma(s^*)} \int_{s^*}^{\infty} \sqrt{\Gamma(u)}du -\alpha^2  \right].
\end{align*}
Our goal is to minimize $K(s^*)$, and we have that the optimal solution $s^{**}$ must satisfy the first-order condition:
\begin{align*}
0=K^{\prime}(s^*)
=\frac{\Gamma^{\prime}(s^*)}{(\Gamma(s^*))^{\frac{3}{2}}}\int_{s^*}^{\infty} \sqrt{\Gamma(s)} ds \left(\frac{\alpha^2}{2}+s^*\Gamma(s^*)-\int_0^{s^*} \Gamma(s) ds\right).\label{step2}
\end{align*}
 Recall that $\Gamma^{\prime}(s^*)< 0$, and thus the first-order condition $K^{\prime}(s^*)=0$ is equivalent to
\begin{align*}
\frac{\alpha^2}{2}+s^*\Gamma(s^*)-\int_0^{s^*} \Gamma(s) ds=0.
\end{align*}

Denote
\begin{align*}
H(s^*):=\frac{\alpha^2}{2}+s^*\Gamma(s^*)-\int_0^{s^*} \Gamma(s) ds, \quad 0<s^*<\infty,
\end{align*}
and we have $H(0)=\alpha^2/2>0$. Noticing that $\int_0^{\infty} \Gamma(s) ds<\infty$, we have $\lim\limits_{s^*\rightarrow \infty} s^*\Gamma(s^*)=0$. Then,
$$\lim\limits_{s^*\rightarrow \infty}   H(s^*)=\frac{\alpha^2}{2}-\int_0^{\infty} \Gamma(s) ds<0,$$
where the last inequality is due to Lemma \ref{amazing}.  In addition,
\begin{align*}
H^{\prime}(s^*)&:=s^* \Gamma^{\prime}(s^*)<0.
\end{align*}
Thus, there exits a unique solution $s^{**}\in (0,\infty)$ for $K^{\prime}(s^*)=0$, which minimizes $K(s^*)$. Then, we can calculate 
\begin{align*}
&\inf_{m\geq 0}\left\{m
\inf_{Q\in\mathcal{M}(\mathbb{R}^{+}): \mathbb{E}^{Q}[N]=m}\left\{\mbox{Var}\left(\int_{0}^{\infty}g(X_s,s)\frac{1_{\{N>s\}}}{Q(N>s)}ds\right)\right\}\right\}\notag\\
&=m^{\ast\ast}  \left[ 2\int_0^{s^{\ast\ast}} \Gamma(s)ds +2\sqrt{\Gamma(s^{\ast\ast})}\int_{s^{\ast\ast}}^{\infty}\sqrt{\Gamma(u)}du -\alpha^2 \right]\notag\\
&=m^{\ast\ast} \left[ \alpha^2+2s^{\ast\ast}\Gamma(s^{\ast\ast})+2\Gamma(s^{\ast\ast})\cdot (m^{\ast\ast}-s^{\ast\ast})-\alpha^2 \right]\notag\\
&=2(m^{\ast\ast})^2 \Gamma(s^{\ast\ast}),
\end{align*}
where the second equality is justified by the definition of  $s^{\ast\ast}$.
This completes the proof.
\end{proof}

\begin{remark}
From the proof of Theorem \ref{main2}, we know $m$ is increasing with respect to $s^*$. Therefore, there exists a unique $m^*$ such that
\begin{align*}
 m^* =s^{\ast\ast} +\frac{\int_{s^{\ast\ast}}^\infty {\sqrt{ \Gamma(s)} ds}}{\sqrt{\Gamma(s^{\ast\ast})}}>\frac{\int_{0}^{\infty}\sqrt{\Gamma(u)}du}{\sqrt{\Gamma(0)}},
\end{align*}
which is the optimal  computational budget for minimizing the work-variance product. Notice that the support of the optimal distribution $Q^*$ is always shifted away from zero for minimizing the work-variance product.
\end{remark}

\begin{corollary}\label{theo3}
	If $\{X_t\}$ is  an exponential L\'evy process with characteristic exponent $\phi$ and $f(x)=x^{\beta}$, then
	\begin{equation*}
	Q^{\ast}(N>s):=
	\begin{cases}
	1 &\mbox{for $0\leq s\leq s^{**}$},\notag
	\\
	e^{\frac{s^{**}-s}{2}|\phi_2(2\beta)|} &\mbox{for $s>s^{**}$},
	\end{cases}
	\end{equation*}
	where $m^*=s^{**}+\frac{2}{|\phi_2(2\beta)|}$ and $\phi_2$ is defined by (\ref{def-phi}), and $s^{**}$ is the unique positive solution to the following transcendental algebraic equation:
	\begin{align*}
	\frac{1}{|\phi_1(\beta)|}+2s^{**} e^{-s^{**}|\phi_2(2\beta)|}-\frac{2}{|\phi_2(2\beta)|}\left(1-e^{-s^{**}|\phi_2(2\beta)|}\right)=0,
	\end{align*}
	which can be solved in a closed form: 
	\begin{align*}
	s^{**}&=W\left(\frac{e}{2} \left(\frac{2}{|\phi_2(2\beta)|}-\frac{1}{|\phi_1(\beta)|}\right)\right)-\frac{1}{|\phi_2(2\beta)|},
	\end{align*}
	where  $W(\cdot)$ is  the Lambert-W function  and $\phi_1$ and $\phi_2$ are defined by (\ref{def-phi}). Furthermore, 
	\begin{align*}
	m^{*}&=W\left(\frac{e}{2} \left(\frac{2}{|\phi_2(2\beta)|}-\frac{1}{|\phi_1(\beta)|}\right)\right)+\frac{1}{|\phi_2(2\beta)|}.
	\end{align*}
\end{corollary}

\begin{proof}
In the case of exponential L\'{e}vy process, we have $\Gamma(s)=\frac{1}{|\phi_1(\beta)|}e^{-s|\phi_2(2\beta)|}$, then the first-order condition $K^{\prime}(s^*)=0$ is equivalent to
\begin{align}
\frac{1}{|\phi_1(\beta)|}+2s^* e^{-s^*|\phi_2(2\beta)|}-\frac{2}{|\phi_2(2\beta)|}\left(1-e^{-s^*|\phi_2(2\beta)|}\right)=0.\label{to-solve}
\end{align}
In order to solve \eqref{to-solve}, we denote $y=s^{**}+\frac{1}{|\phi_2(2\beta)|}$, then we can rewrite the algebraic equation into the following equivalent form:
\begin{align}
y\cdot e^y&=\frac{e}{2} \left(\frac{2}{|\phi_2(2\beta)|}-\frac{1}{|\phi_1(\beta)|}\right),\label{to-solve2}
\end{align}
and note that the right-hand side is positive due to the L\'{e}vy-Khintchine theorem and Jensen's inequality, i.e., $2|\phi_1(\beta)|>|\phi_2(2\beta)|$ always holds. By the definition of the Lambert-W function, we can recognize that the solution to \eqref{to-solve2} is given explicitly by $y=W(b)$, 
and here $b:=\frac{e}{2} \left(\frac{2}{|\phi_2(2\beta)|}-\frac{1}{|\phi_1(\beta)|}\right)$. Then we have the desired solutions of $s^{**}$ and $m^*$.  Note that $b>0$, and the Lambert-W function is uniquely defined, then this establishes the uniqueness of $s^{**}$. 
Applying the results of Theorem \ref{theo2} completes the proof. 
\end{proof}

\section{Randomization Vs. Fixed Truncation}\label{s3}%of Randomized Estimator With Fixed Truncation Estimator}

As discussed in the last section, randomization inevitably increases the variance, although it eliminates the bias. Obviously, small bias and variance are desirable in practice. Basically, whether the optimal randomization is favorable or not depends on the tradeoff between bias and variance.
Thus, we consider a utility function as follows:
\begin{align*}
U_w(I_m):=-(\mathbb{E}[I_m]-\alpha)^2-w Var(I_m),\quad w\geq 0,
\end{align*}
where $I_m$ is an estimator of $\alpha$ subject to the computational budget $m$. A large $w$ indicates more weight on the variance and less weight on the bias in the tradeoff of these two factors.
We denote the optimal randomized estimator with computational budget $m$ as $I_m^r$ and the fixed truncation estimator with computational budget $m$ as $I_m^f$ defined by
\begin{align}
I_m^f:=\int_0^m g(X_s,s) ds.\notag
\end{align}
Then, we have the following result.

\begin{proposition}\label{p1} For any $w>0$, when $m$ is sufficiently small,
$$U_w(I_m^r)<U_w(I_m^f)~.$$
\end{proposition}
\begin{proof}
Note that
\begin{align*}
&U_w(I_m^f)=-\left(\mathbb{E}\int_{m}^{\infty}g(X_{s},s)ds\right)^{2}-w\text{Var}\left(\int_{0}^{m}g(X_{s},s)ds\right),
\\
&U_w(I_m^r)=-w \inf_{Q\in\mathcal{M}(\mathbb{R}^{+}):\mathbb{E}^{Q}[N]=m}
\text{Var}\left(\int_{0}^{N}\frac{g(X_{s},s)}{Q(N>S)}ds\right).
\end{align*}
From Theorem \ref{main2}, we have that for  $m\leq \int_{0}^{\infty}\sqrt{\Gamma(u)}du/\sqrt{\Gamma(0)}$,
 $Q^{\ast}(N>s)=\frac{m\sqrt{\Gamma(s)}}{\int_{0}^{\infty}\sqrt{\Gamma(u)}du}$ such that
\begin{align*}
U_w(I_m^r) &=-2 w\int_{0}^{\infty}\frac{\Gamma(s)}{Q^{\ast}(N>s)}ds
+w\left(\int_{0}^{\infty}\mathbb{E}[g(X_{s},s)]ds\right)^{2}
\\
&=-\frac{2 w}{m}\left(\int_{0}^{\infty}\sqrt{\Gamma(u)}du\right)^{2}
+w\left(\int_{0}^{\infty}\mathbb{E}[g(X_{s},s)]ds\right)^{2}.
\nonumber
\end{align*}

When $m$ is sufficiently small, the inequality $m\leq \int_{0}^{\infty}\sqrt{\Gamma(u)}du/\sqrt{\Gamma(0)}$
is satisfied and $U_w(I_m^r)\rightarrow-\infty$ as $m\rightarrow 0$, and
on the other hand, as $m\rightarrow 0$,
we have $U_w(I_m^f)\rightarrow-\left(\mathbb{E}\int_{0}^{\infty}g(X_{s},s)ds\right)^{2}>-\infty$, which proves the conclusion.
 \end{proof}
\begin{remark} The proposition indicates that as long as the variance is of  concern to a practitioner, the optimal randomized estimator would not be favored if the computational budget is small enough.
When $m$ is small, the distribution of $N$
must be very skewed in order to make $\mathbb{E}^{Q}[N]=m$
and the estimator unbiased at the same time.
Specifically, $Q(N>s)$ is small for $s>m$,
which leads to a very large variance, because $Q(N>s)$ appears in the denominator
of the expression $2\int_{0}^{\infty}\frac{\Gamma(s)}{Q(N>s)}ds$.
\end{remark}

In general, we define the following threshold level:
\begin{align*}
w(m):=\frac{\left(\mathbb{E}\int_{m}^{\infty}g(X_{s},s)ds\right)^{2}}{ \inf\limits_{Q\in\mathcal{M}(\mathbb{R}^{+}):\mathbb{E}^{Q}[N]=m}
\mbox{Var}\left(\int_{0}^{N}\frac{g(X_{s},s)}{Q(N>s)}ds\right)-\mbox{Var}\left(\int_{0}^{m}g(X_{s},s)ds\right)},\label{threshold}
\end{align*}
such that $U_w(I_m^r)>U_w(I_m^f)$ for all $0<w\leq w(m)$.  This threshold $w(m)$ represents the maximum weight that a practitioner can put onto the variance such that the optimal randomized estimator is more favorable than the fixed truncation estimator. Similar to the proof in Proposition \ref{p1}, it is straightforward to show that $\lim_{m\to0} w(m)=0$. As discussed previously,
\begin{align*}
\inf\limits_{Q\in\mathcal{M}(\mathbb{R}^{+}):\mathbb{E}^{Q}[N]=m}
\mbox{Var}\left(\int_{0}^{N}\frac{g(X_{s},s)}{Q(N>s)}ds\right)
>\mbox{Var}\left(\int_{0}^{\infty}g(X_{s},s)ds\right).\end{align*}
In the case where $\int_{0}^{m}g(X_{s},s)ds$ and $\int_{m}^{\infty}g(X_{s},s)ds$ are positively correlated, then $\mbox{Var}\left(\int_{0}^{\infty}g(X_{s},s)ds\right)>\mbox{Var}\left(\int_{0}^{m}g(X_{s},s)ds\right)$, which implies that $w(m)$ is well defined and non-negative.

%Next, we establish for the exponential L\'{e}vy process that the MSE of the optimal randomized estimator is smaller than the MSE of the fixed truncation estimator when $m$ is large enough. % and $m\leq\int_{0}^{\infty}\sqrt{\Gamma(u)}du/\sqrt{\Gamma(0)}$
%still holds, which requires $g(X_{s},s)$ to decay slowly in time $s$.
%and thus has heavy tails.
\begin{proposition}\label{p2} If $\{X_t\}$ is  an exponential L\'evy process with characteristic exponent $\phi$ and $f(x)=x^{\beta}$, then
\begin{align}
w(m)= \notag
\begin{cases}
\frac{\Delta_1}{\Delta_2-\Delta_4}, \quad 0<m\leq \frac{2}{|\phi_2(2\beta)|},\\
\frac{\Delta_1}{\Delta_3-\Delta_4}, \quad m>\frac{2}{|\phi_2(2\beta)|},
\end{cases}\notag
\end{align}
where 
\begin{align*}
&\Delta_1:=\frac{e^{-2|\phi_{1}(\beta)|m}}{|\phi_{1}(\beta)|^{2}},\\
&\Delta_2:= \frac{8}{m} \frac{1}{|\phi_1(\beta)||\phi_2(2\beta)|^2}-\frac{1}{|\phi_1(\beta)|^2},\\
&\Delta_3:= \frac{2+2e^{-m|\phi_2(2\beta)|+2}}{|\phi_1(\beta)||\phi_2(2\beta)|}-\frac{1}{|\phi_1(\beta)|^2},\\
& \Delta_4:=\frac{2(1-e^{-m|\phi_{2}(2\beta)|})}{|\phi_{1}(\beta)||\phi_{2}(2\beta)|}
+\frac{2e^{-m|\phi_{2}(2\beta)|}-2e^{-m|\phi_{1}(\beta)|}}{|\phi_{1}(\beta)|(|\phi_{2}(2\beta)|-|\phi_{1}(\beta)|)}
 -\left(\frac{1-e^{-m|\phi_{1}(\beta)|}}{|\phi_{1}(\beta)|}\right)^{2},
\end{align*}
and $\phi_1$ and $\phi_2$ are given by (\ref{def-phi}). 
\end{proposition}
\begin{proof}
Recall that
\begin{equation}
\mathbb{E}[g(X_{t},t)g(X_{s},s)]=e^{(t-s)\phi_{1}(\beta)+s\phi_{2}(2\beta)}.\notag
\end{equation}
and
\begin{equation}
\Gamma(s)=\int_{s}^{\infty}e^{(t-s)\phi_1(\beta)+s\phi_2(2\beta)}dt
=\frac{1}{|\phi_1(\beta)|}e^{-s|\phi_2(2\beta)|}.\notag
\end{equation}
We have
\begin{equation}
\left(\mathbb{E}\int_{m}^{\infty}g(X_{s},s)ds\right)^{2}
=\left(\int_{m}^{\infty}e^{\phi_{1}(\beta)s}ds\right)^{2}
=\Delta_1,\notag
\end{equation}
and
\begin{equation}
\left(\mathbb{E}\int_{0}^{\infty}g(X_{s},s)ds\right)^{2}
=\frac{1}{|\phi_{1}(\beta)|^{2}}.\notag
\end{equation}
In addition,
\begin{align*}
\mbox{Var}\left(\int_{0}^{m}g(X_{s},s)ds\right)
&=\mathbb{E}\left[\left(\int_{0}^{m}g(X_{s},s)ds\right)^{2}\right]
-\left(\mathbb{E}\int_{0}^{m}g(X_{s},s)ds\right)^{2}
\\
&=2\mathbb{E}\left[\int_{0}^{m}\int_{s}^{m}g(X_{s},s)g(X_{t},t)dtds\right]
-\left(\mathbb{E}\int_{0}^{m}g(X_{s},s)ds\right)^{2}
\\
&=2\int_{0}^{m}\int_{s}^{m}e^{(t-s)\phi_{1}(\beta)+s\phi_{2}(2\beta)}dtds
-\left(\int_{0}^{m}e^{s\phi_{1}(\beta)}ds\right)^{2}
=\Delta_4.
\end{align*}

For the optimal randomized estimator, from the result in Corollary 2, we have two cases.
 If $m>2/|\phi_2(2\beta)|$, then the optimal $Q^{\ast}$ is given by
\begin{equation*}
Q^{\ast}(N>s):=
\begin{cases}
1 &\mbox{for $0\leq s\leq m-\frac{2}{|\phi_2(2\beta)|}$},
\\
e^{\frac{m}{2}|\phi_2(2\beta)|-1-\frac{s}{2}|\phi_2(2\beta)|} &\mbox{for $s>m-\frac{2}{|\phi_2(2\beta)|}$},
\end{cases}
\end{equation*}
and\small
\begin{align*}
\mbox{Var}\left(\int_{0}^{N}\frac{g(X_{s},s)}{Q(N>s)}ds\right)
&=2\int_0^{\infty}\frac{\Gamma(s)}{Q^{\ast}(N>s)}ds-\left(\mathbb{E}\int_{0}^{\infty}g(X_{s},s)ds\right)^{2}\notag\\
&=2\int_0^{m-\frac{2}{|\phi_2(2\beta)|}}\frac{\Gamma(s)}{Q^{\ast}(N>s)}ds+2\int_{m-\frac{2}{|\phi_2(2\beta)|}}^{\infty}\frac{\Gamma(s)}{Q^{\ast}(N>s)}ds-\frac{1}{|\phi_1(\beta)|^2}\notag\\
&=2\int_0^{m-\frac{2}{|\phi_2(2\beta)|}}\Gamma(s)ds+2\int_{m-\frac{2}{|\phi_2(2\beta)|}}^{\infty}\frac{\Gamma(s)}{Q^{\ast}(N>s)}ds-\frac{1}{|\phi_1(\beta)|^2} \notag\\
&=\frac{2\left(1-e^{-|\phi_2(2\beta)| m +2}\right)}{|\phi_1(\beta)| |\phi_2(2\beta)|}  + \frac{4   e^{-|\phi_2(2\beta)| m +2}}{|\phi_1(\beta)| |\phi_2(2\beta)|}-\frac{1}{|\phi_1(\beta)|^2}\notag\\
&=\frac{2+2e^{-m|\phi_2(2\beta)|+2}}{|\phi_1(\beta)||\phi_2(2\beta)|}-\frac{1}{|\phi_1(\beta)|^2}.
\end{align*}\normalsize
If $m\leq 2/|\phi_2(2\beta)|$, then the optimal $Q^{\ast}$ is given by
\begin{equation}
Q^{\ast}(N>s)
=\frac{m}{2}|\phi_2(2\beta)|e^{-\frac{s}{2}|\phi_2(2\beta)|}, \notag
\end{equation}
and
\begin{align*}
\mbox{Var}\left(\int_{0}^{N}\frac{g(X_{s},s)}{Q(N>s)}ds\right)
&=2\int_0^{\infty}\frac{\Gamma(s)}{Q^{\ast}(N>s)}ds-\left(\mathbb{E}\int_{0}^{\infty}g(X_{s},s)ds\right)^{2}\notag\\
&=2\int_0^{\infty}\frac{\Gamma(s)}{Q^{\ast}(N>s)}ds-\frac{1}{|\phi_1(\beta)|^2}\notag=\Delta_2~.
\end{align*}
Then, it is straightforward to prove the conclusion.
\end{proof}

Next we offer an explicit expression for $w(m)$ for 
the Cox-Ingersoll-Ross (CIR) process, which is an affine stochastic process not belonging to the class of exponential L\'{e}vy processes.
The CIR process is governed by the following SDE:
\begin{align*}
dX_t&=\kappa(\theta-X_t)dt+\sigma \sqrt{X_t}dW_t,
\end{align*}
where $W_t$ is a standard Brownian motion. The CIR process is mean reverting to $\theta$, and $\kappa$ governs the speed of the mean reversion. According to the calculation in the appendix, we have 
\begin{align*}
\Gamma(s)&= A e^{-2cs}+ B e^{-(\kappa +2c )s} +C e^{-2(\kappa+c)s},
\end{align*}
where
\begin{align*}
A&:= \frac{\theta^2}{c}+\frac{\theta\sigma^2}{2\kappa(\kappa+c)},    \notag\\
B&:=  \frac{1}{\kappa+c }(X_0-\theta)\left(\theta+\frac{\sigma^2}{\kappa}\right) +\frac{\theta(X_0-\theta)}{c},  \notag\\
C&:= \frac{1}{\kappa+c }\left( (\theta-X_0)^2 +\frac{\sigma^2}{2\kappa}(\theta-2X_0) \right).
\end{align*}

\begin{proposition}\label{cir2} For the CIR process, we have
	\begin{align*}
	w(m)
	=
	\begin{cases}
	\frac{\left( \frac{\theta}{c}e^{-cm}+\frac{X_0-\theta}{c+\kappa}e^{-(c+\kappa)m} \right)^2}{\frac{2}{m}\left(\int_0^{\infty} \sqrt{\Gamma(u)}du\right)^2-\alpha^2-\Delta'_2},   \text{ if } 0<m\leq \frac{\int_{0}^{\infty
	}\sqrt{\Gamma(u)}du}{\sqrt{\Gamma(0)}};\\
	\frac{\left( \frac{\theta}{c}e^{-cm}+\frac{X_0-\theta}{c+\kappa}e^{-(c+\kappa)m} \right)^2}{\Delta'_1-\Delta'_2},   \text{ if } m> \frac{\int_{0}^{\infty
	}\sqrt{\Gamma(u)}du}{\sqrt{\Gamma(0)}},
	\end{cases}
	\end{align*}
	where
	\begin{align*}
\Delta'_1&:= 2\left(A \frac{1-e^{-2cs^{\ast}}}{2c}+ B \frac{1-e^{-(\kappa +2c )s^{\ast}}}{\kappa +2c} +C \frac{1-e^{-2(\kappa+c)s^{\ast}}}{2(\kappa+c)}\right)\notag\\
	&\quad \quad +2(m-s^{\ast})(A e^{-2cs^{\ast}}+ B e^{-(\kappa +2c )s^{\ast}} +C e^{-2(\kappa+c)s^{\ast}})-\alpha^2;\notag\\
	\Delta'_2&:= \frac{\theta^2}{c}(1-e^{-cm})^2+ (X_0-\theta)\left(\theta+\frac{\sigma^2}{\kappa}\right)\frac{2}{c+\kappa}\\
	&\quad \times \left(\frac{1-e^{-(2c+\kappa)m}}{2c+\kappa}-\frac{e^{-(c+\kappa)m}-e^{-(2c+\kappa)m}}{c}\right)\notag\\
	&\quad + \frac{2\theta(X_0-\theta)}{c} \left(\frac{1-e^{-(2c+\kappa)m}}{2c+\kappa}-\frac{e^{-cm}-e^{-(2c+\kappa)m}}{\kappa+c}\right)\notag\\
	&\quad + \left( (\theta-X_0)^2 +\frac{\sigma^2}{2\kappa}(\theta-2X_0) \right) \frac{(1-e^{-(\kappa+c)m})^2}{\kappa+c}
	\\
	&\quad \quad +\frac{\theta\sigma^2}{\kappa(c+\kappa)}\left(\frac{1-e^{-2c m}}{2c}-\frac{e^{-(c+\kappa)m}-e^{-2cm}}{c-\kappa}\right)
	-\left( (X_0-\theta)\frac{1-e^{-(c+\kappa)m}}{c+\kappa} +\theta \frac{1-e^{-cm}}{c} \right)^{2},
	\end{align*}
	and $\Gamma(0)=A+B+C$. Here
 $s^{\ast}$ is the unique positive solution to the following equation:
	\begin{align*}
	s^\ast +\frac{\int_{s^\ast}^\infty {\sqrt{ \Gamma(s)} ds}}{\sqrt{\Gamma(s^\ast)}}&=m.
	\end{align*}
\end{proposition}

The proof of Proposition \ref{cir2} can be found in the Appendix. 

\subsection*{MSE Comparison for Exponential L\'{e}vy Process}

Define
\begin{align*}
MSE_1&:=Var(I_m)+(\mathbb{E}[I_m]-\alpha))^2,\notag\\
MSE_2&:=Var(I_N).
\end{align*}
The randomized estimator is unbiased, so its MSE is just  its variance.

\begin{proposition}\label{compare}
For the exponential L\'{e}vy process, for all $0<m<\infty$,
 $$MSE_1<MSE_2.$$
\end{proposition}

The proof of Proposition \ref{compare} can be found in the Appendix. 

\section{Numerical Experiments\label{s5}}

%\subsection{Example of GBM Process}
In this section, we present numerical experiments illustrating the main theoretical results. We choose the form of the discounted cost to be $g(X_s, s)=e^{-c s}f(X_s)$ with $c>0$ denoting the continuous discount rate. Here $f(X_s)$ refers to the \textit{cost} or \textit{reward} at time $s$, and we adopt the form $f(x):=x^{\beta}$ motivated from the power utility function commonly used in decision theory and economics. As for the choice of the underlying stochastic process $X$, we consider geometric Brownian motion (GBM) and the Cox-Ingersoll-Ross (CIR) process, which were
also considered by \cite{rhee2015unbiased} in their numerical experiments. 

We first consider the  example of discounted cost $g(X_s, s)=e^{-c s}X_s^{\beta}$ ($c>0$) with $X_s$ being the geometric Brownian motion (GBM) model, which is a special case of the exponential L\'{e}vy process.  The GBM model is governed by the following SDE:
\begin{align}
dX_t&=\mu X_t dt+\sigma X_t dW_t, \quad X_0=x_0.\notag
\end{align}
Then, we have  $\phi(\beta)=\left(\mu-\frac{\sigma^2}{2}\right)\beta+\frac{\sigma^2}{2}\beta^2$. 
In addition, we can compute 
\begin{align}
\Gamma(s)=\frac{x_0^{2\beta}}{|\phi_1(\beta)|}e^{-s|\phi_2(2\beta)|},\notag
\end{align}
and
\begin{align*}
\alpha=\int_0^{\infty} e^{-cs} x_0^{\beta}\mathbb{E}[ e^{ s \phi(\beta)}  ] ds=\frac{x_0^{\beta}}{c-\phi(\beta)}=\frac{x_0^{\beta}}{|\phi_1(\beta)|}.
\end{align*}

From Corollary \ref{theo3}, we know that the optimal randomization distribution is a shifted exponential distribution. The survival function of the shifted exponential distribution family is given by
\begin{equation*}
Q(N>s)=
\begin{cases}
1 &\mbox{for $s\leq \delta$},
\\
e^{-\eta (s-\delta)} &\mbox{for $s>\delta$}.
\end{cases}
\end{equation*}

The variance-work product of the shifted exponential distribution family can be expressed as a function of $\delta$ and $\eta$:
\begin{align*}
p(\delta,\eta)&:=\mbox{Var}\left(\int_{0}^{\infty}g(X_s,s)\frac{1_{\{N>s\}}}{Q(N>s)}ds\right)
\cdot\mathbb{E}^{Q}[N]\\
&=\left(2\int_0^{\infty}\frac{\Gamma(s)}{Q(N>s)}ds -\alpha^2\right)\cdot \int_0^{\infty} Q(N>s)ds\notag\\
&=\left(2x_0^{\beta}\frac{1-e^{-\delta |\phi_2(2\beta)|}}{|\phi_1(\beta)||\phi_2(2\beta)|}+2x_0^{\beta}\frac{e^{-\delta |\phi_2(2\beta)|}}{|\phi_1(\beta)|(|\phi_2(2\beta)|-\eta)}-\alpha^2\right)\cdot\left(\delta+\frac{1}{\eta}\right).
\end{align*}
Let $\eta^*:=|\phi_2(2\beta)|/2$, which is the rate in the optimal randomization distribution.  We   set the parameters by $x_0=1$, $\mu=0.1$, $\sigma=0.35$, $c=0.6$, $\beta=0.5$.
In this case, $\alpha=1.7689$, $\eta^*=0.55$, $s^{\ast\ast}=4.7971$, and the minimum work-variance product value is  $p(s^{\ast\ast}, \eta^*)=0.68358$.

\begin{figure}[tb]
\centering
\includegraphics[scale=.22]{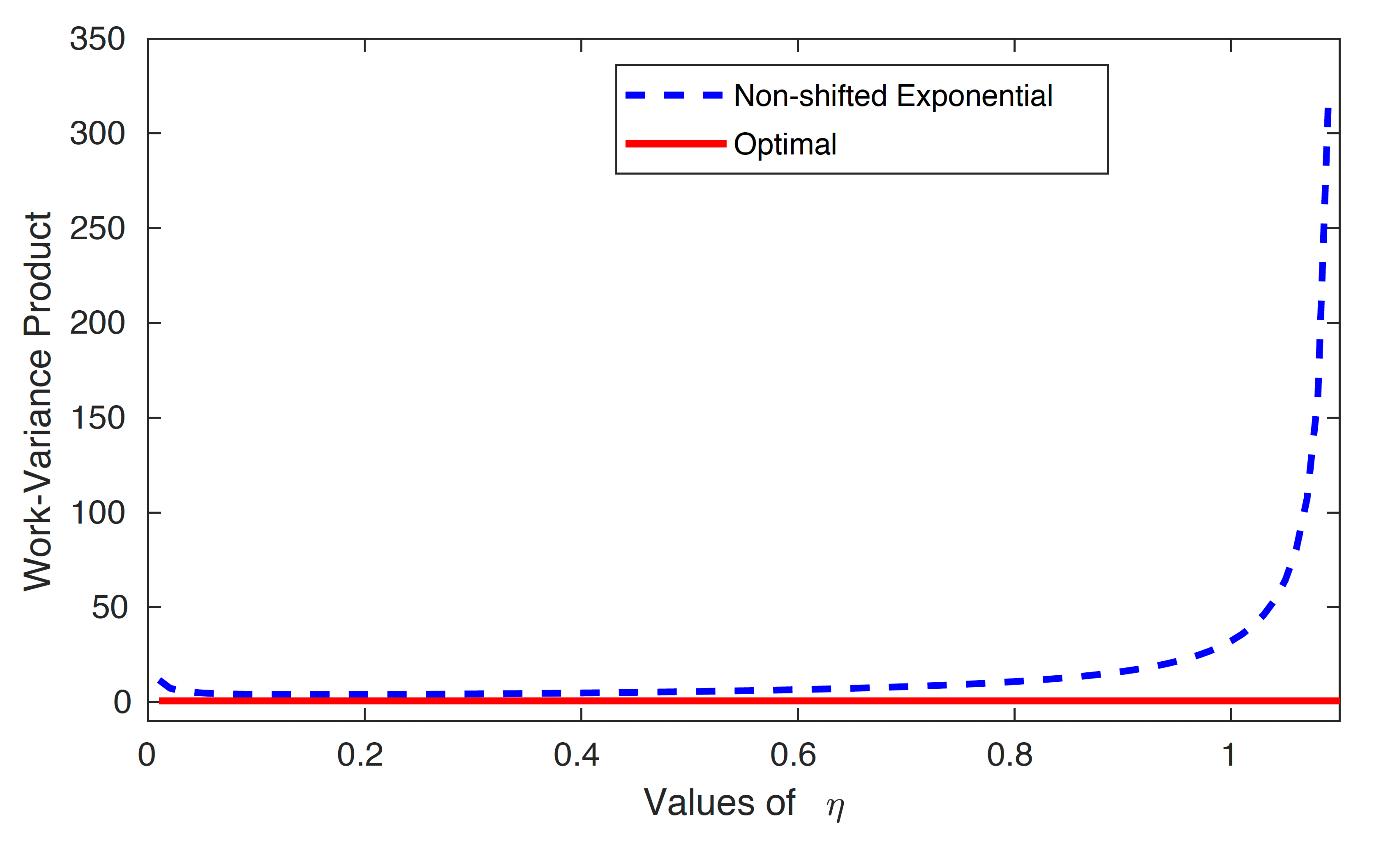}
\caption{Comparison of non-shifted exponential to optimal distribution by varying $\eta$}
\label{fig:comparison1}
\end{figure}

In Figure \ref{fig:comparison1}, the red line is the minimum work-variance product, and the blue line is the work-variance product function of a non-shifted exponential distribution, i.e., $p(0,\eta)$ with $\eta \in [0.01, |\phi_2(2\beta)|]$. We can see the variance-work product of the non-shifted exponential distribution family is strictly larger than the minimum work-variance product, which is consistent with the Theorem \ref{theo2} result that the support of the optimal randomized distribution is always shifted away from zero. The work-variance product increases tremendously when $\eta$ grows larger than $0.6$. A large $\eta$ would lead to a light tail for the survival function, which causes a large variance.

%\newpage
\text{ }
\begin{figure}[]
\centering
         \includegraphics[scale=.85]{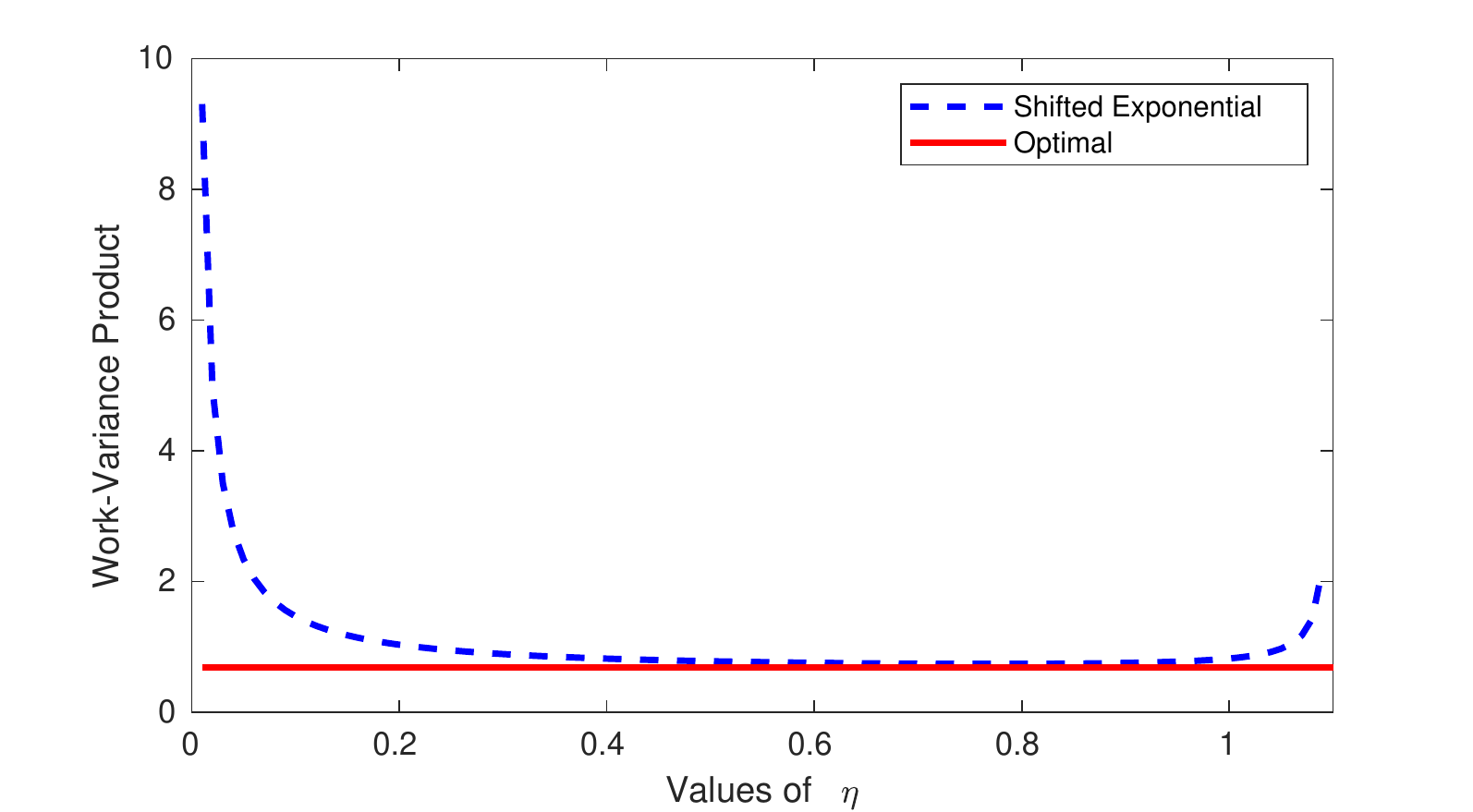}
          \includegraphics[scale=.85]{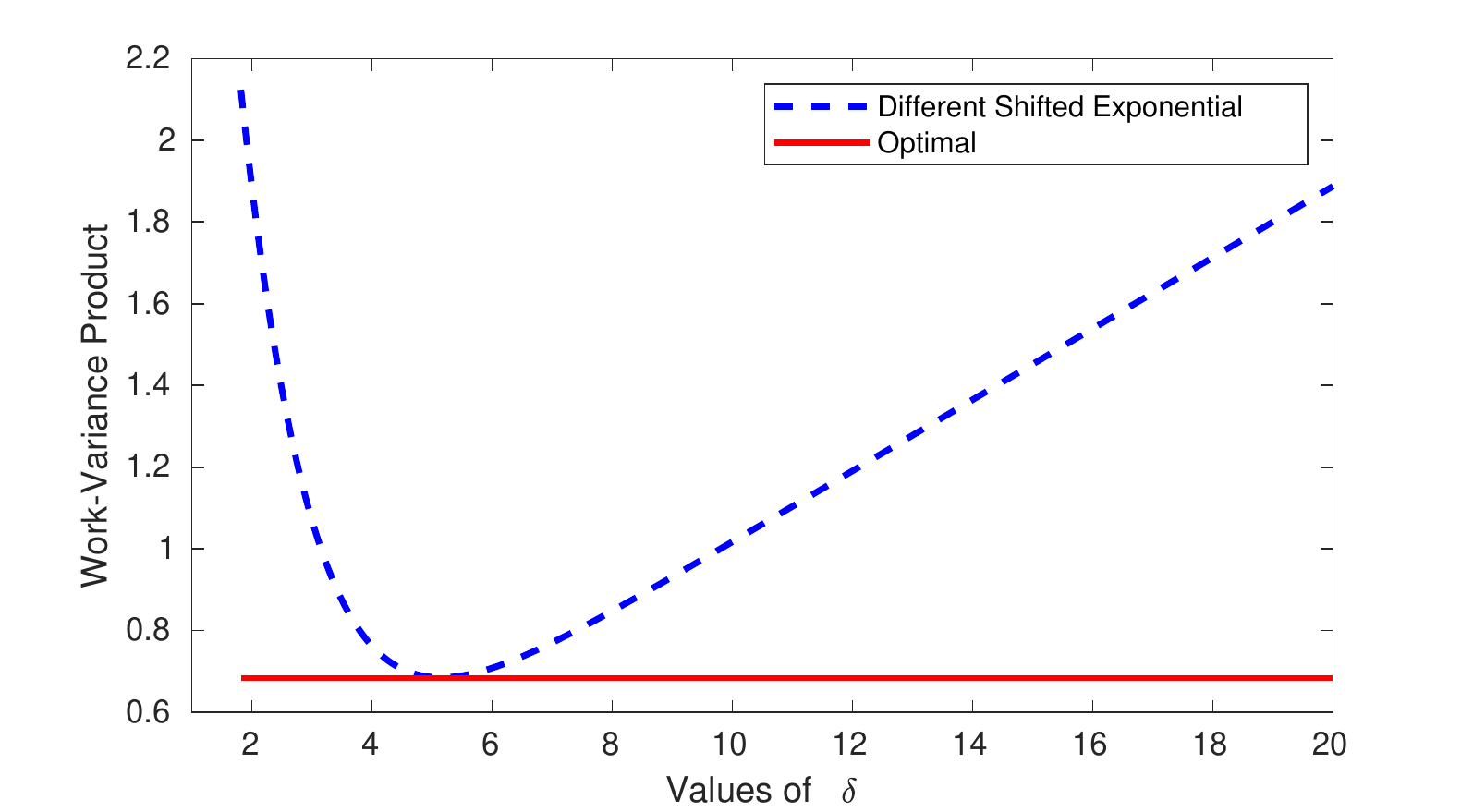}
         \caption{Comparison of shifted exponential to the optimal distribution by varying $\eta$ and $\delta$}
          \label{fig:comparison2}
\end{figure}

%\newpage

In Figure \ref{fig:comparison2}, the blue line in the top graph plots the work-variance product function  $p(s^*,\eta)$
with $\eta \in [0.01, |\phi_2(2\beta)|]$, while the blue line in the bottom graph is the work-variance product function $p(\delta,\eta^*)$ with $\delta \in [2/|\phi_2(2\beta)|, 20]$. We can see the two work-variance product functions deviate from the optimal value  $p(s^*,\eta^*)$ except at the optimal point. The right panel in Figure \ref{fig:comparison2} also substantiates the uniqueness of $s^*$ in Theorem \ref{theo2}.

Then, we plot the threshold level $w(m)$ given by Proposition \ref{p2} for this example. From Figure \ref{fig:comparison4}, we can see that even the optimal weight for all computational budget $m$, which is the level of threshold $w(m)$, is less than $0.13$. This indicates that the advantage of the optimal randomized estimator over the fixed truncation estimator can only be justified under the scenario where the bias is the paramount concern.

\begin{figure}[]
\centering
         \includegraphics[scale=.85]{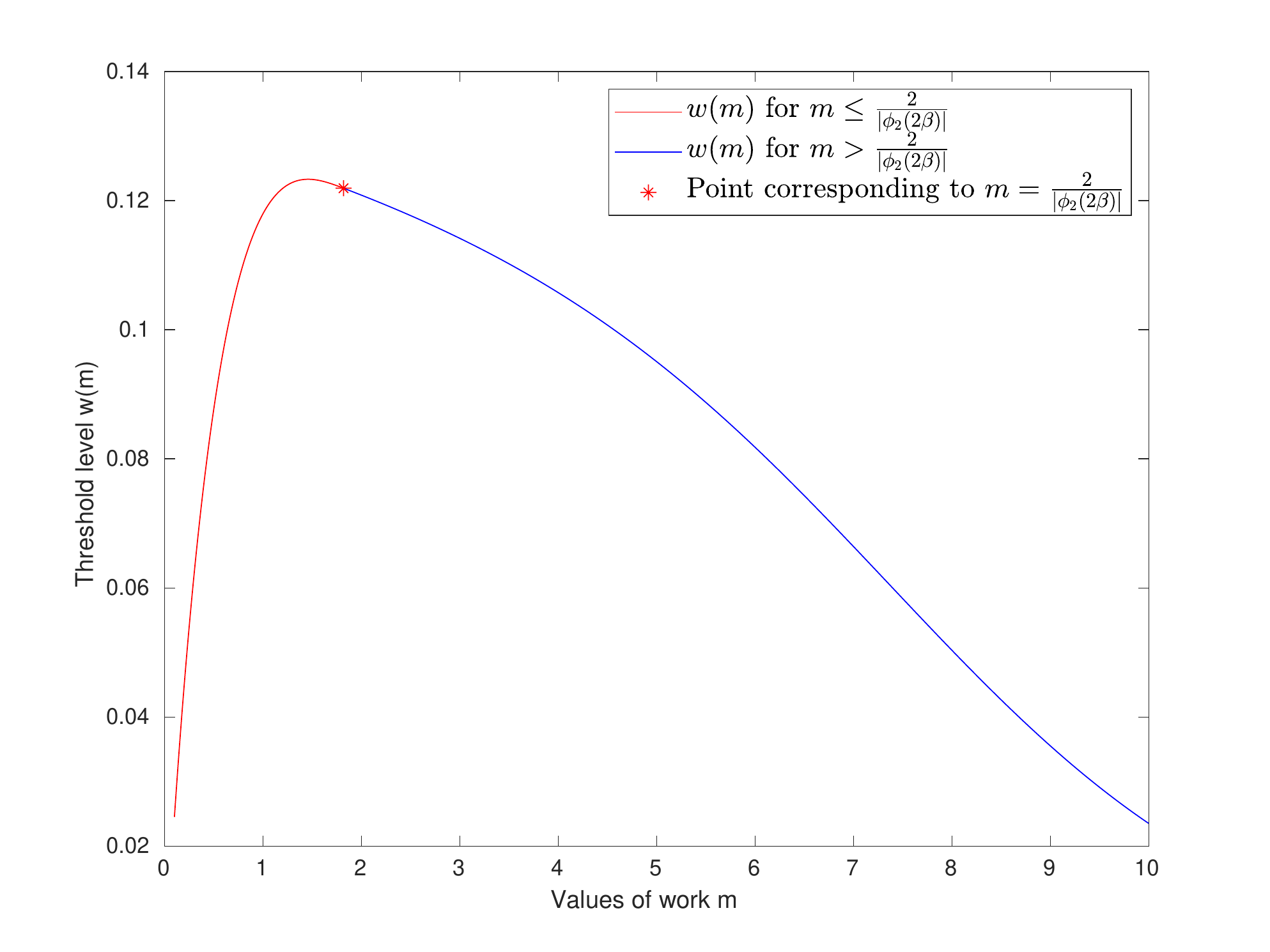}
         \caption{MSE comparison between the fixed truncation estimator and the optimal randomized estimator}
            \label{fig:comparison4}
\end{figure}

For $w=1$ in $U_w(I_m)$, this utility corresponds to the mean squared error (MSE), which is a widely used metric for the efficiency of an estimator. Figure \ref{fig:comparison4} implies that the MSE of the optimal randomized estimator is always larger than the MSE of the fixed truncation estimator. Moreover, we prove in the Appendix that this conclusion holds for all exponential L\'{e}vy processes.
The numerical results for the CIR process can be found in the appendix, which are similar to those for the exponential L\'{e}vy process.

\section{Conclusion\label{s6}}

In this paper, we propose a randomized unbiased estimator for simulating an expected 
cumulative cost/reward.
We derive an explicit form for the optimal distribution of an unbiased randomized estimator
balancing the trade-off between variance and computational cost. 
The optimal distributions are in a shifted distribution class.
To the best of the authors' knowledge, this is the first work resulting in explicit forms for the optimal randomization distribution.
For a discounted continuous cumulative cost contingent on an exponential L\'{e}vy process, the optimal randomization distributions are shifted exponential distributions.
The explicit structure of the distribution function of the optimal random truncation level is particularly useful for ``post-estimation" analysis, and allows us to carry out a full diagnosis of the bias-variance tradeoff.
Moreover, we justify the advantage of the optimal randomized estimator via a utility function taking both bias and  variance into consideration.
Our results are limited to simulating expected cumulative cost/reward.
Future research lies in deriving the optimal randomized distribution and threshold level in the utility function for more general 
RUMC problems. Optimal control theory offers a new perspective to address such RUMC problems.

%%%%%%%%%%%%%%%%%%%%%%%%%%%%%%%%%%
\section*{Acknowledgments.}
We are grateful to the editor and two anonymous referees for their
careful reading of the paper and very helpful suggestions.
This work was supported in part by the National Science Foundation (NSF) under Grants CMMI-1362303, CMMI-1434419, and 
DMS-1613164, by the National Science Foundation of China (NSFC) under Grants 71901003, 71720107003, 91846301, 71790615, 71690232, and by 
the Air Force of Scientific Research (AFOSR) under Grant FA9550-15-10050.

%%%%%%%%%%%%%%%%%%%%%%%%%%%%%%%%%%%%
\bibliographystyle{elsarticle-harv} 
\bibliography{random}

%\newpage

%%%%%%%%%%%%%%%%%%%%%%%%%%%%%%%%%%%%%

\section*{Appendix}
\subsection*{Optimal Randomization for Non-Monotone $\Gamma(s)$}

Let $z(s):=Q(N>s)$. Then, finding the optimal distribution can be viewed as finding the state corresponding to the following optimal control problem:
\begin{align*}&\sup_{u(s)\in(-\infty, 0]}\int_0^\infty -\left(\frac{2\Gamma(s)}{z(s)}+\lambda z(s)\right) ds\\
s.t.~& \dot{z}(s)=u(s),\quad z(0)=1,\quad \lim_{t\to\infty} z(t)=0.
\end{align*}
Notice that the constraint $u(s)\in(-\infty, 0]$ makes sure the state $z(s)$  is non-increasing.
The maximum principle gives a necessary condition of the optimal control:
\begin{align*} &\dot{p}(s)=-\frac{2\Gamma(s)}{z^2(s)}+\lambda,\\
&u^*(s)=\arg \max_{-\infty<u\leq 0} H(z(s),u(s),p(t),s),
\end{align*}
where 
$$ H(z(s),u(s),p(t),s):=-\left(\frac{2\Gamma(s)}{z(s)}+\lambda z(s)\right)+p(s) u.$$
The necessary and sufficient condition for the optimal control can be given by the following Hamilton-Jacobi-Bellman (HJB) partial differential equation:
\begin{align*}\dot{V}(z,s)+\min_u \left\{ \nabla_z V(z,s) u +\left(\frac{2\Gamma(s)}{z}+\lambda z\right)\right\}=0,
\end{align*}
subject to the terminal condition
$\lim_{t\to\infty}V(x,t)=0$.

%\subsection*{MSE Comparison for Exponential L\'{e}vy Process}
%
%Define
%\begin{align*}
%MSE_1&:=Var(I_m)+(\mathbb{E}[I_m]-\alpha))^2,\notag\\
%MSE_2&:=Var(I_N).
%\end{align*}
%Notice that the randomized estimator is unbiased. Thus, the MSE of the randomized estimator is its variance.
%
%\begin{proposition}\label{compare}
%For the exponential L\'{e}vy process, we have that for all $0<m<\infty$,
% $$MSE_1<MSE_2.$$
%\end{proposition}
\subsection*{Proof of Proposition \ref{compare}}

\begin{proof}
 If $m>2/|\phi_2(2\beta)|$,
 we have
\begin{align*}
\text{MSE}_{1}-\text{MSE}_{2}
&=\frac{2(1-e^{-m|\phi_{2}(2\beta)|})}{|\phi_{1}(\beta)||\phi_{2}(2\beta)|}
+\frac{2e^{-m|\phi_{2}(2\beta)|}-2e^{-m|\phi_{1}(\beta)|}}{|\phi_{1}(\beta)|(|\phi_{2}(2\beta)|-|\phi_{1}(\beta)|)}
\\
&\quad -\left(\frac{1-e^{-m|\phi_{1}(\beta)|}}{|\phi_{1}(\beta)|}\right)^{2}
+\frac{e^{-2|\phi_{1}(\beta)|m}}{|\phi_{1}(\beta)|^{2}}+\frac{1}{|\phi_{1}(\beta)|^{2}}
-\frac{2+2e^{-m|\phi_2(2\beta)|+2}}{|\phi_1(\beta)||\phi_2(2\beta)|}.
\end{align*}
Note that we can expand the term
\begin{align*}
\left(\frac{1-e^{-m|\phi_{1}(\beta)|}}{|\phi_{1}(\beta)|}\right)^{2}
&=\frac{1}{|\phi_{1}(\beta)|^{2}}-\frac{2e^{-m|\phi_1(\beta)|}}{|\phi_{1}(\beta)|^{2}}+\frac{e^{-2m|\phi_1(\beta)|}}{|\phi_{1}(\beta)|^{2}}.
\end{align*}

Plugging in this expression into the above, we  can further simplify the above expression to
\begin{align*}
&\text{MSE}_{1}-\text{MSE}_{2}
\\
&=\frac{2(1-e^{-m|\phi_{2}(2\beta)|})}{|\phi_{1}(\beta)||\phi_{2}(2\beta)|}
+\frac{2e^{-m|\phi_{2}(2\beta)|}-2e^{-m|\phi_{1}(\beta)|}}{|\phi_{1}(\beta)|(|\phi_{2}(2\beta)|-|\phi_{1}(\beta)|)}
+\frac{2e^{-m|\phi_1(\beta)|}}{|\phi_{1}(\beta)|^{2}}
-\frac{2+2e^{-m|\phi_2(2\beta)|+2}}{|\phi_1(\beta)||\phi_2(2\beta)|}\notag\\
&=\frac{2}{|\phi_1(\beta)|} \left(  \frac{1-e^{-m|\phi_{2}(2\beta)|}}{|\phi_{2}(2\beta)|}
+\frac{e^{-m|\phi_{2}(2\beta)|}-e^{-m|\phi_{1}(\beta)|}}{|\phi_{2}(2\beta)|-|\phi_{1}(\beta)|}
+\frac{e^{-m|\phi_1(\beta)|}}{|\phi_{1}(\beta)|} -\frac{1+e^{-m|\phi_2(2\beta)|+2}}{|\phi_2(2\beta)|} \right)\notag\\
&=\frac{2}{|\phi_1(\beta)|} \left(  \frac{-e^{-m|\phi_{2}(2\beta)|}-e^{-m|\phi_2(2\beta)|+2}}{|\phi_{2}(2\beta)|}
+\frac{e^{-m|\phi_{2}(2\beta)|}-e^{-m|\phi_{1}(\beta)|}}{|\phi_{2}(2\beta)|-|\phi_{1}(\beta)|}
+\frac{e^{-m|\phi_1(\beta)|}}{|\phi_{1}(\beta)|} \right).
\end{align*}

Then we group the above terms by those related to $e^{-m|\phi_1(\beta)|}$ and those related to $e^{-m|\phi_{2}(2\beta)|}$, and we have
\begin{align*}
\text{MSE}_{1}-\text{MSE}_{2}
&=\frac{2}{|\phi_1(\beta)|} \left( e^{-m|\phi_{2}(2\beta)|} \left(  \frac{-1-e^2}{|\phi_2(2\beta)|} +\frac{1}{|\phi_{2}(2\beta)|-|\phi_{1}(\beta)|}  \right) \right.\notag\\
 &\left. \quad \quad +  e^{-m|\phi_1(\beta)|} \left(  \frac{1}{|\phi_1(\beta)|}-\frac{1}{|\phi_{2}(2\beta)|-|\phi_{1}(\beta)|} \right)\right)
\end{align*}

From the L\'{e}vy-Khintchine formula and Jensen's inequality, it is easy to establish that $|\phi_2(2\beta)|<2|\phi_1(\beta)|$ always holds. Thus we have
\begin{align*}
&\frac{1+e^2}{|\phi_2(2\beta)|}>\frac{1}{|\phi_1(\beta)|},
\end{align*}
or equivalently we have
\begin{align*}
&\frac{-1-e^2}{|\phi_2(2\beta)|}<-\frac{1}{|\phi_1(\beta)|}.
\end{align*}

Using this fact, we have
\begin{align}
\text{MSE}_{1}-\text{MSE}_{2}
&=\frac{2}{|\phi_1(\beta)|} \left( e^{-m|\phi_{2}(2\beta)|} \left(  \frac{-1-e^2}{|\phi_2(2\beta)|} +\frac{1}{|\phi_{2}(2\beta)|-|\phi_{1}(\beta)|}  \right) \right.\notag\\
 &\left. \quad \quad +  e^{-m|\phi_1(\beta)|} \left(  \frac{1}{|\phi_1(\beta)|}-\frac{1}{|\phi_{2}(2\beta)|-|\phi_{1}(\beta)|} \right)\right)\notag\\
 &< \frac{2}{|\phi_1(\beta)|} \left( e^{-m|\phi_{2}(2\beta)|} \left(  -\frac{1}{|\phi_1(\beta)|} +\frac{1}{|\phi_{2}(2\beta)|-|\phi_{1}(\beta)|}  \right) \right.\notag\\
 &\left. \quad \quad +  e^{-m|\phi_1(\beta)|} \left(  \frac{1}{|\phi_1(\beta)|}-\frac{1}{|\phi_{2}(2\beta)|-|\phi_{1}(\beta)|} \right)\right)\notag\\
 &=\frac{2}{|\phi_1(\beta)|} \left( \frac{1}{|\phi_1(\beta)|}-   \frac{1}{|\phi_{2}(2\beta)|-|\phi_{1}(\beta)|}   \right)
 \cdot\left( e^{-m|\phi_1(\beta)|} - e^{-m|\phi_{2}(2\beta)|} \right).\label{eq1}
\end{align}

We further divide the discussion into two cases:
\begin{enumerate}
\item[] If $|\phi_{2}(2\beta)|-|\phi_{1}(\beta)|>0$, then clearly we have $\frac{1}{|\phi_1(\beta)|}-   \frac{1}{|\phi_{2}(2\beta)|-|\phi_{1}(\beta)|} <0$, and $e^{-m|\phi_1(\beta)|} - e^{-m|\phi_{2}(2\beta)|}>0$, thus the right hand side of (\ref{eq1}) is negative.
\item[] If $|\phi_{2}(2\beta)|-|\phi_{1}(\beta)|<0$, then clearly we have $\frac{1}{|\phi_1(\beta)|}-   \frac{1}{|\phi_{2}(2\beta)|-|\phi_{1}(\beta)|} >0$, and $e^{-m|\phi_1(\beta)|} - e^{-m|\phi_{2}(2\beta)|}<0$, thus the right hand side of (\ref{eq1}) is negative.
\end{enumerate}

Above all, we have proved that we always have $\text{MSE}_{1}-\text{MSE}_{2}<0$.

\vspace{0.5cm}

 If $m\leq 2/|\phi_2(2\beta)|$,  we have
\begin{align*}
\text{MSE}_{1}-\text{MSE}_{2}
&=\frac{2(1-e^{-m|\phi_{2}(2\beta)|})}{|\phi_{1}(\beta)||\phi_{2}(2\beta)|}
+\frac{2e^{-m|\phi_{2}(2\beta)|}-2e^{-m|\phi_{1}(\beta)|}}{|\phi_{1}(\beta)|(|\phi_{2}(2\beta)|-|\phi_{1}(\beta)|)}\\
&\qquad -\left(\frac{1-e^{-m|\phi_{1}(\beta)|}}{|\phi_{1}(\beta)|}\right)^{2}
+\frac{e^{-2|\phi_{1}(\beta)|m}}{|\phi_{1}(\beta)|^{2}}
 +\frac{1}{|\phi_{1}(\beta)|^{2}}
-\frac{8}{m}\frac{1}{|\phi_1(\beta)| |\phi_2(2\beta)|^2}.
\end{align*}
Note that we can expand the term
\begin{align*}
\left(\frac{1-e^{-m|\phi_{1}(\beta)|}}{|\phi_{1}(\beta)|}\right)^{2}
&=\frac{1}{|\phi_{1}(\beta)|^{2}}-\frac{2e^{-m|\phi_1(\beta)|}}{|\phi_{1}(\beta)|^{2}}+\frac{e^{-2m|\phi_1(\beta)|}}{|\phi_{1}(\beta)|^{2}}.
\end{align*}

Plugging in this expression into the above, we  can further simplify the above expression to \small
\begin{align*}
\text{MSE}_{1}-\text{MSE}_{2}
&=\frac{2(1-e^{-m|\phi_{2}(2\beta)|})}{|\phi_{1}(\beta)||\phi_{2}(2\beta)|}
+\frac{2e^{-m|\phi_{2}(2\beta)|}-2e^{-m|\phi_{1}(\beta)|}}{|\phi_{1}(\beta)|(|\phi_{2}(2\beta)|-|\phi_{1}(\beta)|)}\\
&\quad +\frac{2e^{-m|\phi_1(\beta)|}}{|\phi_{1}(\beta)|^{2}}
-\frac{8}{m}\frac{1}{|\phi_1(\beta)| |\phi_2(2\beta)|^2}\notag\\
&=\frac{2}{|\phi_1(\beta)|} \left(  \frac{1-e^{-m|\phi_{2}(2\beta)|}}{|\phi_{2}(2\beta)|}
+\frac{e^{-m|\phi_{2}(2\beta)|}-e^{-m|\phi_{1}(\beta)|}}{|\phi_{2}(2\beta)|-|\phi_{1}(\beta)|}\right.\\
&\qquad \qquad \qquad \left.+\frac{e^{-m|\phi_1(\beta)|}}{|\phi_{1}(\beta)|}
-\frac{4}{m |\phi_2(2\beta)|^2} \right)\notag\\
&<\frac{2}{|\phi_1(\beta)|} \left(  \frac{1-e^{-m|\phi_{2}(2\beta)|}}{|\phi_{2}(2\beta)|}
+\frac{e^{-m|\phi_{2}(2\beta)|}-e^{-m|\phi_{1}(\beta)|}}{|\phi_{2}(2\beta)|-|\phi_{1}(\beta)|}\right.\\
&\left.\qquad \qquad \qquad +\frac{e^{-m|\phi_1(\beta)|}}{|\phi_{1}(\beta)|}
-\frac{2}{ |\phi_2(2\beta)|} \right)\notag\\
&=\frac{2}{|\phi_1(\beta)|} \left(  \frac{-1-e^{-m|\phi_{2}(2\beta)|}}{|\phi_{2}(2\beta)|}
+\frac{e^{-m|\phi_{2}(2\beta)|}-e^{-m|\phi_{1}(\beta)|}}{|\phi_{2}(2\beta)|-|\phi_{1}(\beta)|}
+\frac{e^{-m|\phi_1(\beta)|}}{|\phi_{1}(\beta)|}\right)\notag\\
&<\frac{2}{|\phi_1(\beta)|} \left(  \frac{-2e^{-m|\phi_{2}(2\beta)|}}{|\phi_{2}(2\beta)|}
+\frac{e^{-m|\phi_{2}(2\beta)|}-e^{-m|\phi_{1}(\beta)|}}{|\phi_{2}(2\beta)|-|\phi_{1}(\beta)|}
+\frac{e^{-m|\phi_1(\beta)|}}{|\phi_{1}(\beta)|}\right),
\end{align*}\normalsize
where in the second last inequality, we have utilized the assumption that $m  |\phi_2(2\beta)|<2$, and in the last inequality we have used the fact that $e^{-m|\phi_{2}(2\beta)|}<1$.

Then we group the above terms by those related to $e^{-m|\phi_1(\beta)|}$ and those related to $e^{-m|\phi_{2}(2\beta)|}$:
\begin{align*}
\text{MSE}_{1}-\text{MSE}_{2}
&<\frac{2}{|\phi_1(\beta)|} \left( e^{-m|\phi_{2}(2\beta)|} \left(  \frac{-2}{|\phi_2(2\beta)|} +\frac{1}{|\phi_{2}(2\beta)|-|\phi_{1}(\beta)|}  \right) \right.\notag\\
 &\left. \quad \quad +  e^{-m|\phi_1(\beta)|} \left(  \frac{1}{|\phi_1(\beta)|}-\frac{1}{|\phi_{2}(2\beta)|-|\phi_{1}(\beta)|} \right)\right)
\end{align*}

From the L\'{e}vy-Khintchine formula and the Jensen inequality, it is easy to establish that $|\phi_2(2\beta)|<2|\phi_1(\beta)|$ always holds. Thus we have
\begin{align*}
&\frac{-2}{|\phi_2(2\beta)|}<-\frac{1}{|\phi_1(\beta)|}.
\end{align*}

Thus we have
\begin{align}
\text{MSE}_{1}-\text{MSE}_{2}
&<
\frac{2}{|\phi_1(\beta)|} \left( e^{-m|\phi_{2}(2\beta)|} \left(  \frac{-2}{|\phi_2(2\beta)|} +\frac{1}{|\phi_{2}(2\beta)|-|\phi_{1}(\beta)|}  \right) \right.\notag\\
 &\left. \quad \quad +  e^{-m|\phi_1(\beta)|} \left(  \frac{1}{|\phi_1(\beta)|}-\frac{1}{|\phi_{2}(2\beta)|-|\phi_{1}(\beta)|} \right)\right)\notag\\
 &<
\frac{2}{|\phi_1(\beta)|} \left( e^{-m|\phi_{2}(2\beta)|} \left(  -\frac{1}{|\phi_1(\beta)|} +\frac{1}{|\phi_{2}(2\beta)|-|\phi_{1}(\beta)|}  \right) \right.\notag\\
 &\left. \quad \quad +  e^{-m|\phi_1(\beta)|} \left(  \frac{1}{|\phi_1(\beta)|}-\frac{1}{|\phi_{2}(2\beta)|-|\phi_{1}(\beta)|} \right)\right)\notag\\
 &=\frac{2}{|\phi_1(\beta)|} \left( \frac{1}{|\phi_1(\beta)|}-   \frac{1}{|\phi_{2}(2\beta)|-|\phi_{1}(\beta)|}   \right)
\cdot\left( e^{-m|\phi_1(\beta)|} - e^{-m|\phi_{2}(2\beta)|} \right).\label{eq2}
 \end{align}

Note that the right hand side of (\ref{eq2}) is exactly the same as the right hand side of (\ref{eq1}), thus following similar arguments, we can establish that $MSE_1-MSE_3<0$. This completes the proof. 
\end{proof}

\subsection*{Derivations for the CIR process}
We can compute that
\begin{align*}
\alpha=\int_0^{\infty}e^{-cs} \E[X_s]ds&=\int_0^{\infty}e^{-cs} \left(X_0 e^{-\kappa s} +\theta (1-e^{-\kappa s}) \right)ds
=\frac{\theta}{c}+\frac{X_0-\theta}{\kappa+c}.
\end{align*}

We have the following expression for its cross moment for $s<t$
\begin{align*}
\E[X_s X_t]=&\theta^2+e^{-\kappa t}(X_0-\theta)\left(\theta+\frac{\sigma^2}{\kappa}\right) +e^{-\kappa s}\theta(X_0-\theta)\\
&+e^{-\kappa(t+s)}\left( (\theta-X_0)^2 +\frac{\sigma^2}{2\kappa}(\theta-2X_0) \right) +\frac{\theta\sigma^2}{2\kappa}e^{-\kappa(t-s)}.
\end{align*}

Define the process $Y_t:=e^{-c t} X_t$, then we have
\begin{align*}
\E[Y_s Y_t]&=e^{-c  (t+s)}\E[X_s X_t]\notag\\
&=\theta^2e^{-c  (t+s)} +e^{-c  s-(\kappa+c ) t}(X_0-\theta)\left(\theta+\frac{\sigma^2}{\kappa}\right)
+e^{-(\kappa+c ) s-c  t}\theta(X_0-\theta)\notag\\
&\quad +e^{-(\kappa+c )(t+s)}\left( (\theta-X_0)^2 +\frac{\sigma^2}{2\kappa}(\theta-2X_0) \right)
+\frac{\theta\sigma^2}{2\kappa}e^{-(c +\kappa)t-(c -\kappa)s}.
\end{align*}

For $f(x)=x$,  we have
\begin{align}
\Gamma(s)&=\int_s^{\infty} \E[e^{-c  s}f(X_s)e^{-c  t} f(X_t)]=\int_s^{\infty} \E[Y_s Y_t]dt\notag\\
&=\frac{\theta^2}{c }e^{-2c  s} +\frac{e^{-(\kappa +2c )s}}{\kappa+c }(X_0-\theta)\left(\theta+\frac{\sigma^2}{\kappa}\right)+\frac{e^{-(\kappa+2c ) s}}{c }\theta(X_0-\theta)\notag\\
& +\frac{e^{-2(\kappa+c )s}}{\kappa+c }\left( (\theta-X_0)^2 +\frac{\sigma^2}{2\kappa}(\theta-2X_0) \right)
+\frac{\theta\sigma^2}{2\kappa}\frac{e^{-2c  s}}{\kappa+c }.\label{gamma_cir}
\end{align}

%We can easily verify that $\Gamma(\cdot)$ satisfies the desired Assumption \ref{as1} under certain assumptions on the parameters.
 Here we have to determine the sufficient conditions to be imposed onto the parameters in order to have the Assumption \ref{as1} to be satisfied. A sufficient condition is given by
\begin{align*}
\theta(X_0-\theta)&>0, \quad (\theta-X_0)^2 +\frac{\sigma^2}{2\kappa}(\theta-2X_0) >0,
\end{align*}
which is equivalent to requiring
\begin{align*}
X_0&>\theta+\frac{\sigma^2}{2\kappa} +\sqrt{\frac{\sigma^2}{2\kappa} \left(\theta+\frac{\sigma^2}{2\kappa}\right)}.
\end{align*}

%We regroup the terms in \eqref{gamma_cir} to obtain 
Then we  calculate the ingredients for the determination of the optimal randomization distribution. For example, in the case of solving the optimization problem (\ref{prob1}), from the result in Theorem \ref{main1}, we have
\begin{equation*}
Q^{\ast}(N>s)=
\begin{cases}
1 &\mbox{$0\leq s\leq s^{\ast}$},
\\
\sqrt{\frac{2(A e^{-2cs}+ B e^{-(\kappa +2c )s} +C e^{-2(\kappa+c)s})}{\lambda}} &\mbox{$s>s^{\ast}$},
\end{cases}
\end{equation*}
where $s^*=\inf\{s\in[0,\infty):~ A e^{-2cs}+ B e^{-(\kappa +2c )s} +C e^{-2(\kappa+c)s}\leq \lambda/2\}$.

In the case of solving the constrained optimization when we are given the expected computational work $m>0$, from the characterization in Theorem \ref{main2}, we have
\begin{equation*}
Q^{\ast}(N>s):=
\begin{cases}
1 &\mbox{ $0\leq s\leq s^{\ast}$},
\\
\sqrt{\frac{A e^{-2cs}+ B e^{-(\kappa +2c )s} +C e^{-2(\kappa+c)s}}{A e^{-2cs^*}+ B e^{-(\kappa +2c )s^*} +C e^{-2(\kappa+c)s^*}}} &\mbox{ $s>s^{\ast}$},
\end{cases}
\end{equation*}
where
\begin{align*}
s^\ast +\frac{\int_{s^\ast}^\infty {\sqrt{ A e^{-2cs}+ B e^{-(\kappa +2c )s} +C e^{-2(\kappa+c)s}} ds}}{\sqrt{A e^{-2cs^*}+ B e^{-(\kappa +2c )s^*} +C e^{-2(\kappa+c)s^*}}}&=m.%\label{root}
\end{align*}

For the minimization of the variance-work product, from Theorem \ref{theo2}, we have
\begin{equation*}
Q^{\ast}(N>s):=
\begin{cases}
1 &\mbox{$0\leq s\leq s^{\ast\ast}$},
\\
\sqrt{\frac{A e^{-2cs}+ B e^{-(\kappa +2c )s} +C e^{-2(\kappa+c)s}}{A e^{-2cs^{\ast\ast}}+ B e^{-(\kappa +2c )s^{\ast\ast}} +C e^{-2(\kappa+c)s^{\ast\ast}}}} &\mbox{ $s>s^{\ast\ast}$},
\end{cases}
\end{equation*}
where  $s^{\ast\ast}$ is the unique positive solution to the following equation:
\begin{align}
&\frac{\alpha^2}{2}+s^{\ast\ast}(A e^{-2cs^{\ast\ast}}+ B e^{-(\kappa +2c )s^{\ast\ast}} +C e^{-2(\kappa+c)s^{\ast\ast}})\notag\\
&-\int_0^{s^{\ast\ast}} (A e^{-2cs}+ B e^{-(\kappa +2c )s} +C e^{-2(\kappa+c)s}) ds=0.\label{simplify2}
\end{align}

We can simplify the equation \eqref{simplify2} as
\begin{align*}
&\frac{\alpha^2}{2}+s^{\ast\ast}(A e^{-2cs^{\ast\ast}}+ B e^{-(\kappa +2c )s^{\ast\ast}} +C e^{-2(\kappa+c)s^{\ast\ast}})\notag\\
&- \left(A \frac{1-e^{-2cs^{\ast\ast}}}{2c}+ B \frac{1-e^{-(\kappa +2c )s^{\ast\ast}}}{\kappa +2c} +C \frac{1-e^{-2(\kappa+c)s^{\ast\ast}}}{2(\kappa+c)}\right)=0.%\label{simplify3}
\end{align*}

To calculate the optimal work-variance product, recall that the optimal level of $m$ is given by
\begin{align*}
m^{\ast\ast}&=s^{\ast\ast} +\frac{\int_{s^{\ast\ast}}^{\infty} \sqrt{A e^{-2cs}+ B e^{-(\kappa +2c )s} +C e^{-2(\kappa+c)s} } ds}{\sqrt{A e^{-2cs^{\ast\ast}}+ B e^{-(\kappa +2c )s^{\ast\ast}} +C e^{-2(\kappa+c)s^{\ast\ast}}}}
\end{align*}

Then the optimal work-variance product in the case of the CIR process is given by
\begin{align*}
&m^{\ast\ast}  \left[ 2\int_0^{s^{\ast\ast}} \Gamma(s)ds +2\sqrt{\Gamma(s^{\ast\ast})}\int_{s^{\ast\ast}}^{\infty}\sqrt{\Gamma(u)}du -\alpha^2 \right]\notag\\
&=m^{\ast\ast}  \left[ \alpha^2+2s^{\ast\ast}\Gamma(s^{\ast\ast})+2\Gamma(s^{\ast\ast})\cdot (m^{\ast\ast}-s^{\ast\ast})-\alpha^2 \right]\notag\\
&=2(m^{\ast\ast})^2 \Gamma(s^{\ast\ast})\notag\\
&=2(m^{\ast\ast})^2  (A e^{-2cs^{\ast\ast}}+ B e^{-(\kappa +2c )s^{\ast\ast}} +C e^{-2(\kappa+c)s^{\ast\ast}}).
\end{align*}
Note that in the second equality above we have utilized the defining equation characterizing $s^{\ast\ast}$. 
%%%%%%%%%%%%%%%%%%%%%%%%%%%%%%%%%%%%%%%%%%%%%%%%%%%%%%%
\subsection*{Proof of Proposition \ref{cir2}} 
\begin{proof}
We have the following calculations:
\begin{align*}
\mathbb{E}\left[\int_{m}^{\infty}g(X_{s},s)ds\right]&=\int_m^{\infty}e^{-cs} \E[X_s]ds\notag\\
&=\int_m^{\infty}e^{-cs} \left(X_0 e^{-\kappa s} +\theta (1-e^{-\kappa s}) \right)ds\notag\\
&=\frac{\theta}{c}e^{-cm}+\frac{X_0-\theta}{c+\kappa}e^{-(c+\kappa)m},
\end{align*}
and 
\begin{align}
\mbox{Var}\left(\int_{0}^{m}g(X_{s},s)ds\right)
&=\mathbb{E}\left[\left(\int_{0}^{m}g(X_{s},s)ds\right)^{2}\right]
-\left(\mathbb{E}\int_{0}^{m}g(X_{s},s)ds\right)^{2}\notag\\
&=2\int_{0}^{m}\int_{s}^{m}  \E[Y_s Y_t] dtds
-\left(\int_{0}^{m}e^{-cs} \left(X_0 e^{-\kappa s} +\theta (1-e^{-\kappa s}) \right)ds\right)^{2}.\label{cir3}
\end{align}

For the first term of the right hand side of \eqref{cir3}, we have\small
\begin{align*}
&2\int_{0}^{m}\int_{s}^{m}  \E[Y_s Y_t] dtds=\frac{\theta^2}{c}(1-e^{-cm})^2+ (X_0-\theta)\left(\theta+\frac{\sigma^2}{\kappa}\right)\frac{2}{c+\kappa}\notag\\
&\quad \times \left(\frac{1-e^{-(2c+\kappa)m}}{2c+\kappa}-\frac{e^{-(c+\kappa)m}-e^{-(2c+\kappa)m}}{c}\right)\notag\\
&\quad + \frac{2\theta(X_0-\theta)}{c} \left(\frac{1-e^{-(2c+\kappa)m}}{2c+\kappa}-\frac{e^{-cm}-e^{-(2c+\kappa)m}}{\kappa+c}\right)\notag\\
&\quad + \left( (\theta-X_0)^2 +\frac{\sigma^2}{2\kappa}(\theta-2X_0) \right) \frac{(1-e^{-(\kappa+c)m})^2}{\kappa+c}
+\frac{\theta\sigma^2}{\kappa(c+\kappa)}\left(\frac{1-e^{-2c m}}{2c}-\frac{e^{-(c+\kappa)m}-e^{-2cm}}{c-\kappa}\right).
\end{align*}\normalsize

The second term on the right hand side of \eqref{cir3} can be calculated as
\begin{align*}
\left(\int_{0}^{m}e^{-cs} \left(X_0 e^{-\kappa s} +\theta (1-e^{-\kappa s}) \right)ds\right)^{2}
=\left( (X_0-\theta)\frac{1-e^{-(c+\kappa)m}}{c+\kappa} +\theta \frac{1-e^{-cm}}{c} \right)^{2}.
\end{align*}

For the optimal randomized distribution, we have
\begin{align*}
\inf\limits_{Q\in\mathcal{M}(\mathbb{R}^{+}):\mathbb{E}^{Q}[N]=m}
\mbox{Var}\left(\int_{0}^{N}\frac{g(X_{s},s)}{Q(N>s)}ds\right)
&=\mbox{Var}\left(\int_{0}^{N}\frac{g(X_{s},s)}{Q^{\ast}(N>s)}ds\right)\notag\\
&=2\int_0^{\infty}\frac{\Gamma(s)}{Q^{\ast}(N>s)}ds-\left(\mathbb{E}\int_{0}^{\infty}g(X_{s},s)ds\right)^{2}\notag\\
&=2\int_0^{s^{\ast}}\frac{\Gamma(s)}{Q^{\ast}(N>s)}ds+2\int_{s^{\ast}}^{\infty}\frac{\Gamma(s)}{Q^{\ast}(N>s)}ds-\alpha^2.
\end{align*}

We divide the discussion into two cases: \\

(i) If  $m> \int_{0}^{\infty
}\sqrt{\Gamma(u)}du/\sqrt{\Gamma(0)}$, then  we have
\begin{align*}
\inf\limits_{Q\in\mathcal{M}(\mathbb{R}^{+}):\mathbb{E}^{Q}[N]=m}
\mbox{Var}\left(\int_{0}^{N}\frac{g(X_{s},s)}{Q(N>s)}ds\right)
&=\mbox{Var}\left(\int_{0}^{N}\frac{g(X_{s},s)}{Q^{\ast}(N>s)}ds\right)\notag\\
&=2\int_0^{s^{\ast}}\frac{\Gamma(s)}{Q^{\ast}(N>s)}ds+2\int_{s^{\ast}}^{\infty}\frac{\Gamma(s)}{Q^{\ast}(N>s)}ds-\alpha^2\notag\\
&=2\int_0^{s^{\ast}}\Gamma(s) ds+2\sqrt{\Gamma(s^{\ast})}\int_{s^{\ast}}^{\infty}\sqrt{\Gamma(u)}du-\alpha^2\notag\\
&=2\int_0^{s^{\ast}}\Gamma(s) ds+2\Gamma(s^{\ast})(m-s^{\ast})-\alpha^2\notag\\
&=2\left(A \frac{1-e^{-2cs^{\ast}}}{2c}+ B \frac{1-e^{-(\kappa +2c )s^{\ast}}}{\kappa +2c} +C \frac{1-e^{-2(\kappa+c)s^{\ast}}}{2(\kappa+c)}\right)\notag\\
&\quad \quad +2(m-s^{\ast})(A e^{-2cs^{\ast}}+ B e^{-(\kappa +2c )s^{\ast}} +C e^{-2(\kappa+c)s^{\ast}})-\alpha^2,
\end{align*}
where we have utilized the characterization equation of $s^{\ast}$ in the part (i) of the (modified) Theorem 2 for the constrained optimization. 

(ii) If  $m\leq  \int_{0}^{\infty
}\sqrt{\Gamma(u)}du/\sqrt{\Gamma(0)}$, then  we have 
\begin{align*}
\inf\limits_{Q\in\mathcal{M}(\mathbb{R}^{+}):\mathbb{E}^{Q}[N]=m}
\mbox{Var}\left(\int_{0}^{N}\frac{g(X_{s},s)}{Q(N>s)}ds\right)
&=\mbox{Var}\left(\int_{0}^{N}\frac{g(X_{s},s)}{Q^{\ast}(N>s)}ds\right)
\\
&=2\int_0^{\infty}\frac{\Gamma(s)}{Q^{\ast}(N>s)}ds-\alpha^2\notag
\\
&=\frac{2}{m}\left(\int_0^{\infty} \sqrt{\Gamma(u)}du\right)^2-\alpha^2.
\end{align*}
Then it is straightforward to prove the conclusion.
\end{proof}

%%%%%%%%%%%%%%%%%%%%%%%%%%%%%%%
\subsection*{Numerical Results for CIR Process}
We consider a CIR process, and use the following parameter sets:  $\kappa=3$, $\theta=0.2$, $\sigma=0.3$, $c=0.6$, $x_0=0.5$. %In Figure \ref{fig:CIR1}, we can see that the non-optimal shifted distributions lead to larger variance-work products than that of the optimal shifted distribution.
 In Figure \ref{fig:CIR-MSE1}, the MSE of the optimal randomized estimator is larger than the MSE of the fixed truncation estimator. In Figure \ref{fig:CIR-w1}, the threshold function $w(m)$ is plotted, and we can see that the threshold function first increases and then decreases with the optimal value less than $0.14$. 
 
 %\newpage

%\begin{figure}[tb]
%	\centering
%	\includegraphics[scale=.50]{CIR_example.pdf}
%	\caption{Comparison of Work-Variance Product of the Optimal shifted estimator to estimators with other shift levels in the case of CIR}
%	\label{fig:CIR1}
%\end{figure}

\begin{figure}[tb]
	\centering
	\includegraphics[scale=.73]{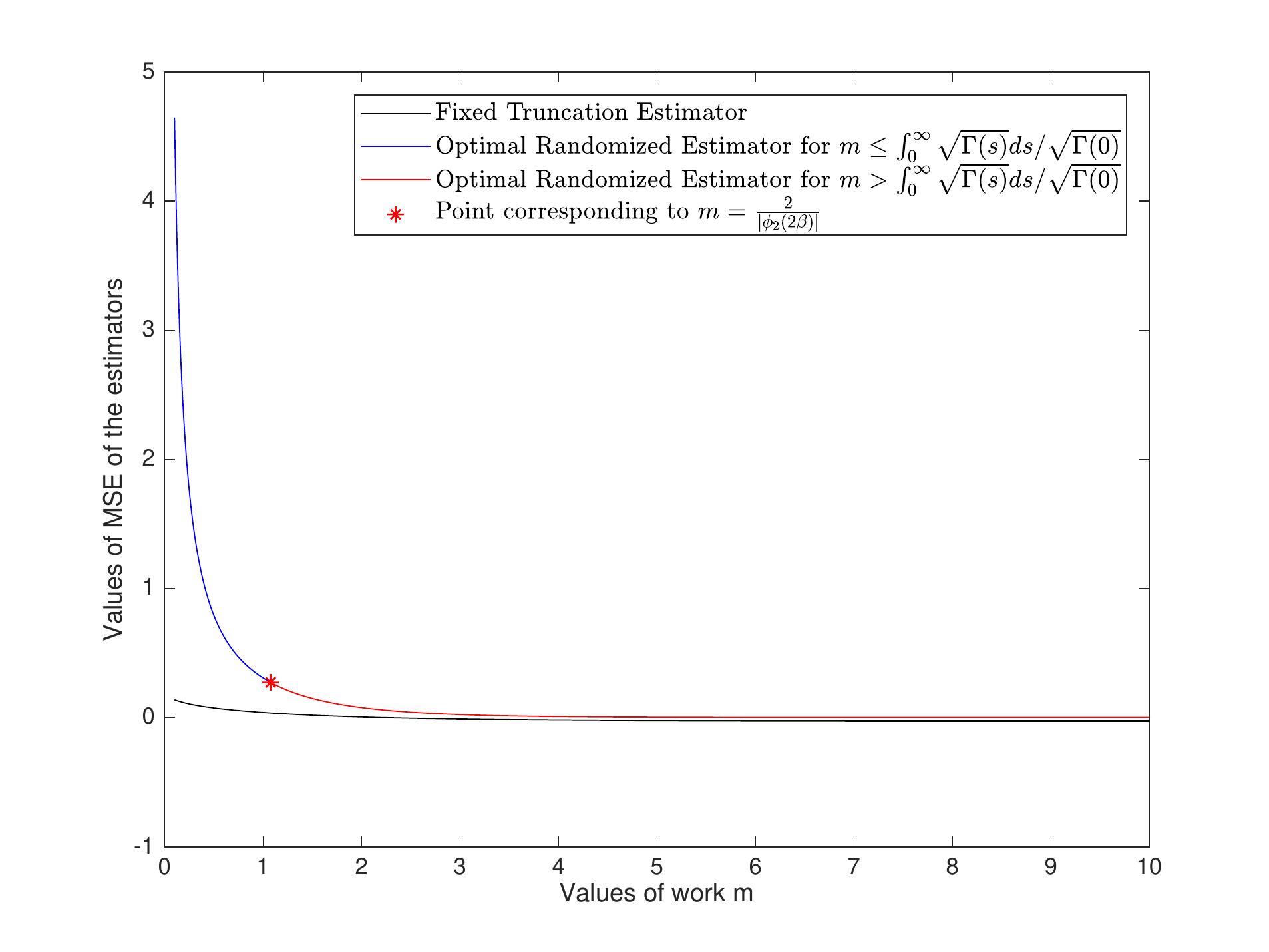}
	\caption{Comparison of MSE of the Optimal randomized estimator to Fixed truncation estimator in the CIR case}
	\label{fig:CIR-MSE1}
\end{figure}

\begin{figure}[tb]
	\centering
	\includegraphics[scale=.78]{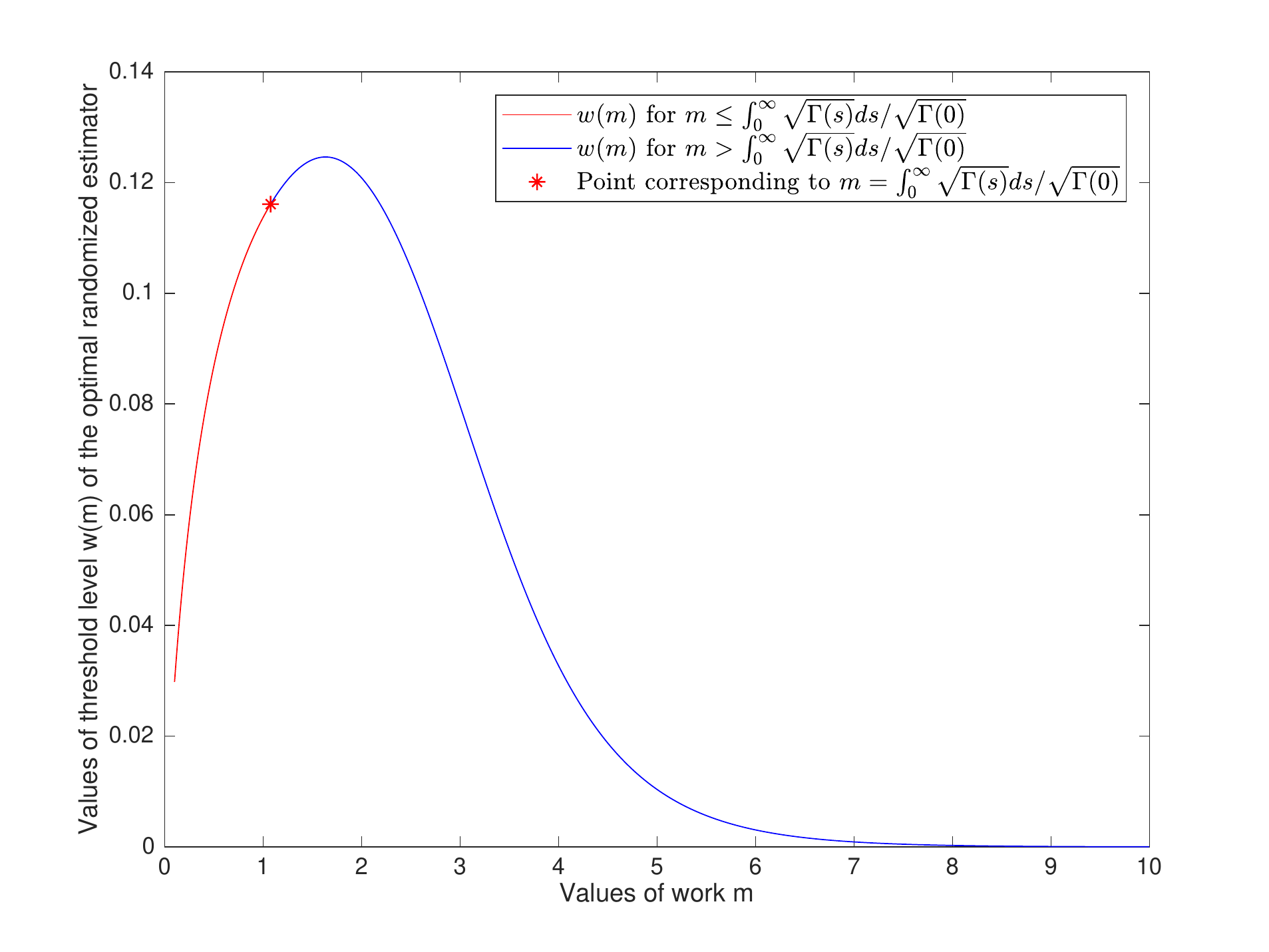}
	\caption{Plot of  the threshold level $w(m)$ of the optimal randomized estimator in the CIR case}
	\label{fig:CIR-w1}
\end{figure}

\end{document}